\documentclass[a4paper,12pt]{article} 
\usepackage{natbib}
\usepackage[finnish,english]{babel}
\usepackage[utf8]{inputenc}
\usepackage[T1]{fontenc}
\usepackage{lmodern}
\usepackage{amsmath} 
\usepackage{amsthm}
\usepackage{amssymb}
\usepackage{bm}
\usepackage{mathtools}
\usepackage{fancyhdr}
\usepackage{mathrsfs}
\usepackage{graphicx} 

\theoremstyle{plain}

\newtheorem{theoremaux}{Theorem}
\newenvironment{theorem}{\medskip\begin{theoremaux}}{\end{theoremaux}}

\newtheorem{propositionaux}[theoremaux]{Proposition}

\newtheorem{lemmaaux}[theoremaux]{Lemma}
\newenvironment{lemma}{\medskip\begin{lemmaaux}}{\end{lemmaaux}}

\newtheorem{corollaryaux}[theoremaux]{Corollary}
\newenvironment{corollary}{\medskip\begin{corollaryaux}}{\end{corollaryaux}}

\newtheorem{conjectureaux}[theoremaux]{Conjecture}

\newtheorem{assumptionaux}[theoremaux]{Assumption}
\newenvironment{assumption}{\medskip\begin{assumptionaux}}{\end{assumptionaux}}

\theoremstyle{definition}

\newtheorem{definitionaux}[theoremaux]{Definition}

\newtheorem{exampleaux}[theoremaux]{Example}

\newenvironment{example}{\medskip\begin{exampleaux}}{\end{exampleaux}\medskip}

\newtheorem{cexampleaux}[theoremaux]{Counterexample}

\newtheorem{exerciseaux}{Exercise}

\newtheorem{problemaux}{Problem}

\theoremstyle{remark}

\usepackage{ifthen}
\newcommand*{\keyterm}[1]{\emph{#1}}
\DeclareMathOperator{\re}{Re}			
\DeclareMathOperator{\im}{Im}			
\newcommand{\ran}{\mathcal{R}}						
\renewcommand{\ker}{\mathcal{N}}					

\DeclareMathOperator*{\dist}{dist}       

\newcommand*{\ldelim}[2]{\csname#1l\endcsname#2}   
\newcommand*{\rdelim}[2]{\csname#1r\endcsname#2}   
\newcommand*{\mdelim}[2]{\csname#1m\endcsname#2}   

\newcommand{\ga}{\alpha}
\newcommand{\gb}{\beta}

\renewcommand{\gg}{\gamma}
\newcommand{\gd}{\delta}

\newcommand{\gl}{\lambda}
\newcommand{\gw}{\omega}
\newcommand{\gs}{\sigma}

\newcommand{\eps}{\varepsilon}

\newcommand*{\C}{{\mathbb{C}}}     

\newcommand*{\R}{{\mathbb{R}}}     
\newcommand*{\Z}{{\mathbb{Z}}}     
\newcommand*{\N}{{\mathbb{N}}}     

\newcommand*{\Lin}{{\mathcal{L}}}   
\newcommand*{\Dom}{{\mathcal{D}}}   

\newcommand*{\abs} [1]{\lvert#1\rvert}
\newcommand*{\norm}[1]{\lVert#1\rVert}

\newcommand*{\set} [1]{\{#1\}}
\newcommand*{\setm}[2]{\{\,#1\mid#2\,\}}   
\newcommand*{\iprod}[2]{\langle#1,#2\rangle}    

\newcommand*{\floor}[1]{\lfloor#1\rfloor}

\newcommand*{\Abs}[2][default]{\ifthenelse{\equal{#1}{default}}{\left\lvert#2\right\rvert}{\ldelim{#1}{\lvert}#2\rdelim{#1}{\rvert}}}
\newcommand*{\Norm}[2][default]{\ifthenelse{\equal{#1}{default}}{\left\lVert#2\right\rVert}{\ldelim{#1}{\lVert}#2\rdelim{#1}{\rVert}}}

\newcommand*{\Iprod}[3][default]{\ifthenelse{\equal{#1}{default}}{\left\langle#2,#3\right\rangle}{\ldelim{#1}{\langle}#2,#3\rdelim{#1}{\rangle}}}
\newcommand*{\Dualpair}[3][default]{\ifthenelse{\equal{#1}{default}}{\left\langle#2,#3\right\rangle}{\ldelim{#1}{\langle}#2,#3\rdelim{#1}{\rangle}}}

\newcommand*{\List}[2][1]{\set{#1,\ldots,#2}}

\newcommand*{\Set}[2][default]{\ifthenelse{\equal{#1}{default}}{\left\{#2\right\}}{\ldelim{#1}{\{}#2\rdelim{#1}{\}}}}

\newcommand*{\dda}[3][1]{\ifthenelse{\equal{#1}{1}}{\frac{d#3}{d#2}}{\frac{d^{#1}#3}{d#2^{#1}}}}
\newcommand*{\ddb}[2][1]{\ifthenelse{\equal{#1}{1}}{\frac{d}{d#2}}{\frac{d^{#1}}{d#2^{#1}}}}
\newcommand*{\pd}[3][1]{\ifthenelse{\equal{#1}{1}}{\frac{\partial{#2}}{\partial{#3}}}{\frac{\partial^{#1}{#2}}{\partial#3^{#1}}}}
\newcommand*{\pdb}[2][1]{\ifthenelse{\equal{#1}{1}}{\frac{\partial}{\partial{#2}}}{\frac{\partial^{#1}}{\partial#2^{#1}}}}

\newcommand*{\conj}[1]{\overline{#1}}

\newcommand*{\Lp}[1][p]{L^{#1}}
\newcommand*{\lp}[1][p]{\ell^{#1}}

\newcommand*{\inv}{^{-1}}

\newcommand{\eq}[1]{\begin{align*}#1\end{align*}}
\newcommand{\eqn}[1]{\begin{align}#1\end{align}}

\newcommand{\hiddenproof}[2]{
\ifthenelse{#1}{Hidden}{#2}
}

\newcommand{\longcomm}[1]{}

\usepackage{datetime}
\newdateformat{findate}{\THEDAY.\THEMONTH.\THEYEAR}

\usepackage[left=1in,top=1in,bottom=1.5in,right=1in]{geometry}

\renewcommand{\subsectionmark}[1]{} 
\theoremstyle{definition}

\newcommand{\Omi}{\mathcal{O}}

\newcommand{\Igw}{I_A} 

\newcommand{\Y}{\C^p}

\newcommand{\Ma}{M_{\tilde{\ga}/\ga}}
\newcommand{\Mb}{M_{\tilde{\gb}/\gb}}
\newcommand{\Mc}{M_{\tilde{\gg}/\gg}}

\bibpunct{[}{]}{,}{n}{}{,}

\title{Robustness of Strong Stability of Semigroups}
\author{Lassi Paunonen\thanks{Tampere University of Technology, PO.Box 553, 33101 Tampere, Finland, \texttt{lassi.paunonen@tut.fi}}}
\date{~}

\begin{document}

\maketitle
\vspace{-8ex}

\thispagestyle{plain}

\begin{abstract}
In this paper we study the preservation of strong stability of strongly continuous semigroups on Hilbert spaces. In particular, we study a situation where the generator of the semigroup has a finite number of spectral points on the imaginary axis and the norm of its resolvent operator is polynomially bounded near these points. We characterize classes of finite rank perturbations preserving the strong stability of the semigroup. In addition, we improve recent results on preservation of polynomial stability of a semigroup under finite rank perturbations of its generator. Theoretic results are illustrated with an example where we consider the preservation of the strong stability of a multiplication semigroup.
\end{abstract}

\smallskip


\section{Introduction}

It is a well-known fact that the exponential stability of a strongly continuous semigroup $T(t)$ is preserved under all sufficiently small perturbations of its infinitesimal generator~$A$. However, in a situation where $T(t)$ is not exponentially stable, but merely \keyterm{strongly stable}, i.e.,
\eq{
\lim_{t\rightarrow \infty}\, \norm{T(t)x} =0,
\qquad \forall x\in X,
} 
no general conditions for the preservation the stability of $T(t)$ are known.
On the contrary, it is acknowledged that strong stability may be extremely sensitive to even arbitrarily small perturbations
of its infinitesimal generator. 

Recently in~\cite{Pau12,Pau13a} it was shown that a subclass of strongly stable semigroups, the so-called \keyterm{polynomially stable semigroups}, do indeed possess good robustness properties. The key observation was that in the case of polynomial stability, the size of the perturbation $A+BC$ should not be measured using the regular operator norms $\norm{B}$ and $\norm{C}$, but instead using graph norms $\norm{(-A)^\gb B}$ and $\norm{(-A^\ast)^\gg C^\ast}$ for suitable exponents $\gb$ and $\gg$.
The polynomially stable semigroups have a 
characteristic property that their generators have no spectrum on the imaginary axis $i\R$. Therefore, many of the strongly stable semigroups encountered in applications are beyond the scope of the perturbation results in~\cite{Pau12,Pau13a}.
In this paper we study the
robustness properties of semigroups whose generators do have spectrum on the imaginary axis. 
In particular, we consider a situation where $A$ has a finite number of spectral points on the imaginary axis, and the norm of the resolvent operator of $A$ is polynomially bounded near these points. 
We show
that the semigroups of this type have surprising robustness properties.

The results presented in this paper again demonstrate that for a strongly stable semigroup $T(t)$, the size of the perturbation should not be measured using the regular operator norm, but instead using suitable graph norms related to the generator $A$. 
Our main results reveal large and easily characterizable classes of finite rank perturbations that preserve the strong stability of $T(t)$.
The results can be applied, for example, in the study of linear partial differential equations, and in the control of infinite-dimensional linear systems.

To the author's knowledge, robustness properties of strong stability of semigroups with spectrum on the imaginary axis have not been studied previously in the literature. Some results on preservation of strong stability of compact semigroups can be found in~\cite{caramanrobstrongstab}.
However, any strongly stable compact semigroup is actually exponentially stable~\cite[Ex. V.1.6(4)]{engelnagel}.

To illustrate our conditions for the preservation of stability, 
we begin by stating our main result in a situation where
$A$ has a single imaginary spectral point $\gs(A)\cap i\R = \set{0}$ belonging to the continuous spectrum of $A$. 
We further assume that there exists $\ga\geq 1$ 
  such that 
  \eq{
  \sup_{0< \abs{\gw}\leq 1} \abs{\gw}^\ga \norm{R(i\gw,A)}<\infty
  } 
  and $\sup_{\abs{\gw}>1} \norm{R(i\gw,A)}<\infty$. These assumptions are satisfied, for example, if $A$ generates a strongly stable analytic semigroup with $0\in\gs_c(A)$, and $\abs{\gl}\norm{R(\gl,A)}\leq M$ outside some sector in $\C^-$.
  Since $A$ is injective and $\ran(A)$ is dense, the operator $-A$ has a densely defined inverse $(-A)\inv$. Furthermore, $(-A)\inv$ and $(-A^\ast)\inv$ are sectorial operators, and thus for $\gb,\gg\geq 0$ the fractional powers $(-A)^{-\gb}: \ran( (-A)^{\gb})\subset X\rightarrow X$ and $(-A^\ast)^{-\gg}: \ran( (-A^\ast)^{\gg})\subset X\rightarrow X$ are well-defined.
We 
consider finite rank perturbations of the form $A+BC$, where $B\in \Lin(\C^p,X)$, and $C\in \Lin(X,\C^p)$ satisfy
  \eqn{
  \label{eq:pertcondintro}
  \ran(B) \subset \ran( (-A)^\gb),
  \qquad
  \ran(C^\ast) \subset \ran( (-A^\ast)^\gg) 
  }
for some $\gb,\gg\geq 0$.
We choose to measure the size of the perturbation $BC$ 
using the graph norms 
$\norm{B} + \norm{(-A)^{-\gb}B}$ and 
$\norm{C} + \norm{(-A^\ast)^{-\gg}C^\ast}$. Theorem~\ref{thm:stabpertintro} shows that
this is exactly the right choice for the purposes of studying the preservation of the strong stability of $T(t)$. 

\begin{theorem} 
  \label{thm:stabpertintro}
  Let $\gb+\gg=\ga$. There exists $\gd>0$ such that if $B$ and $C$ satisfy~\eqref{eq:pertcondintro} and 
\eq{
\norm{B}+\norm{(-A)^{-\gb}B}<\gd, \quad \mbox{and} \quad \norm{C}+ \norm{(-A^\ast)^{-\gg} C^\ast}<\gd,
}
  then the semigroup generated by $A+BC$ is strongly stable.
\end{theorem}

In Section~\ref{sec:mainresults} we state Theorem~\ref{thm:stabpertintro} in a more general situation where $\gs(A)\cap i\R = \set{i\gw_k}_{k\in\Igw}$ for a finite set $\Igw$ of indices. In this case, the preservation of stability requires that for all $k\in\Igw$ the graph norms
$\norm{B} + \norm{(i\gw_k-A)^{-\gb}B}$ and 
$\norm{C} + \norm{(-i\gw_k-A^\ast)^{-\gg}C^\ast}$ are sufficiently small.

In addition to studying the preservation of strong stability, we also improve the results concerning robustness of polynomial stability presented in~\cite{Pau12,Pau13a}. In these references it was shown that the polynomial stability of a semigroup generated by $A$ is preserved under a finite rank perturbation $A+BC$ if for some $\gb,\gg\geq 0$ satisfying $\gb+\gg\geq \ga$ we have
\eqn{
\label{eq:BCpolstabrancondintro}
  \ran(B)\subset \Dom( (-A)^\gb), \quad \mbox{and} \quad \ran(C^\ast)\subset \Dom( (-A^\ast)^\gg) ,
}
and if the graph norms $\norm{(-A)^\gb B}$ and $\norm{(-A^\ast)^\gg C^\ast}$ are small enough. 
However, in these results one of the associated exponents $\gb,\gg\geq 0$ was required to an integer, or alternatively, larger than or equal to $\ga$. 
The techniques used in this paper allow us to remove these restrictions on the exponents. In particular, we show that for arbitrary exponents $\gb,\gg\geq 0$ satisfying $\gb+\gg\geq \ga$ the polynomial stability of a semigroup generated by $A$ is preserved provided that the perturbation satisfies~\eqref{eq:BCpolstabrancondintro} and the corresponding graph norms are small enough.

The paper is organized as follows. In Section~\ref{sec:mainresults} we state our main results on the preservation of strong and polynomial stability. The result on robustness of strong stability is proved in parts throughout Sections~\ref{sec:specpert} and~\ref{sec:stabpert}. Section~\ref{sec:polpert} contains the proof of the result on the preservation of polynomial stability. In Section~\ref{sec:multex} we illustrate the theoretic results with an example. In particular, we study the robustness properties of a strongly stable multiplication semigroup. Section~\ref{sec:conclusions} contains concluding remarks.

We conclude this section by applying Theorem~\ref{thm:stabpertintro} to study the preservation of the strong stability of a semigroup generated by a bounded diagonal operator.

\begin{example}
  Let $X=\lp[2](\C)$ and define $A\in \Lin(X)$ by
  \eq{
  A= \sum_{k=1}^\infty -\frac{1}{k} \iprod{\cdot}{e_k}e_k
  }
  where $e_k$ are the natural basis vectors. The operator generates a strongly stable semigroup $T(t)$ and satisfies $\gs(A)\cap i\R = \set{0}\subset \gs_c(A)$. Since for $\gw\neq 0$ we have $\norm{R(i\gw,A)} = \dist(i\gw,\gs(A))^{-1} = \abs{\gw}^{-1}$, the assumptions of Theorem~\ref{thm:stabpertintro} are satisfied for $\ga=1$.
  The operator $-A$ has an unbounded self-adjoint inverse, and for $\gb\geq 0$ its fractional powers 
are given by
  \eq{
  (-A)^{-\gb} x = \sum_{k=1}^\infty k^\gb \iprod{x}{e_k}e_k, 
  \qquad 
  x\in
  \ran( (-A)^{\gb})
  = \Bigl\{x\in X \,\Bigm|\,\sum_{k=1}^\infty k^{2\gb} \abs{\iprod{x}{e_k}}^2<\infty\Bigr\}.
  }
  If we consider a rank one perturbation $A + \iprod{\cdot}{c}b$ with $b,c\in X$, then Theorem~\ref{thm:stabpertintro} in particular states that the semigroup generated by the perturbed operator is strongly stable if $\norm{b}$ and $\norm{c}$ are small, and for some $\gb,\gg\geq 0$ satisfying $\gb+\gg=1$ the norms
  \eq{
  \norm{(-A)^{-\gb} b}^2 = \sum_{k=1}^\infty k^{2\gb} \abs{\iprod{b}{e_k}}^2
  \qquad \mbox{and} \qquad
  \norm{(-A^\ast)^{-\gg} c}^2 = \sum_{k=1}^\infty k^{2\gg} \abs{\iprod{c}{e_k}}^2
  }
  are finite and small.
\end{example}

If~$X$ and~$Y$ are Banach spaces and~$A:X\rightarrow Y$ is a linear operator, we denote by~$\Dom(A)$, $\ran(A)$, and $\ker(A)$ the domain, the range, and the kernel of~$A$, respectively. 
The space of bounded linear operators from~$X$ to~$Y$ is denoted by~$\Lin(X,Y)$. 
If \mbox{$A:\Dom(A)\subset X\rightarrow X$,} then~$\gs(A)$, $\gs_p(A)$, $\gs_c(A)$ and~$\rho(A)$ denote the spectrum, the point spectrum, the continuous spectrum and the \mbox{resolvent} set of~$A$, respectively. For~$\gl\in\rho(A)$ the resolvent operator is given by \mbox{$R(\gl,A)=(\gl -A)^{-1}$}. 
The inner product on a Hilbert space is denoted by $\iprod{\cdot}{\cdot}$.

For a function $f:\R \rightarrow \R$ and for $\ga\geq 0$ we use the notation 
\eq{
f(\gw)=\Omi \left( \abs{\gw}^\ga \right)
}
if there exist constants $M>0$ and $\gw_0\geq 0$ such that
$\abs{f(\gw)}\leq M \abs{\gw}^\ga$
for all $\gw\in \R$ with $\abs{\gw}\geq \gw_0$.

\section{Main Results}
\label{sec:mainresults}

In this section we present our main results. 
It is well-known that if the semigroup generated by $A$ is strongly stable, then $A$ may have no eigenvalues on the imaginary axis, and therefore operators $A-i\gw$ are injective for all $\gw\in\R$.
Moreover, since $X$ is a Hilbert space, the Mean Ergodic Theorem~\cite{arendtbatty} applied to operators $A-i\gw$ shows that
\eq{
X = \ker(A-i\gw) \oplus \overline{\ran(A-i\gw)} = \overline{\ran(A-i\gw)}.
}
Therefore, the part of the spectrum of $A$ that is on the imaginary axis belongs to the continuous spectrum.

In the following we formulate our assumptions on the unperturbed operator $A$ as well as on the components $B$ and $C$ of the perturbing operator. The main assumption is that the intersection $\gs(A)\cap i\R =\set{i\gw_k}_{k\in\Igw}$ is finite, and the norm of the resolvent operator is polynomially bounded near the points $i\gw_k$.

\begin{assumption} 
  \label{ass:Astandass} 
  Let $X$ be a Hilbert space. Assume that the operators $A: \Dom(A)\subset X\rightarrow X$, $B\in \Lin(\Y,X)$, and $C \in \Lin(X,\Y)$ satisfy the following for some $\ga\geq 1$, $\gb,\gg\geq 0$, and $M_A>0$.
  \begin{itemize}
    \item[\textup{1.}]
      The operator $A$ generates a strongly stable semigroup, and  $\gs(A)\cap i\R = \set{i\gw_k}_{k\in \Igw}$ for a finite set $\Igw$ of indices and $d_A=\min_{k\neq l}\abs{\gw_k-\gw_l}>0$.
      Moreover, there exists $0<\eps_A\leq \max \set{1,d_A/3}$ such that
      \eqn{
        \label{eq:Aresolventgrowthorder}
	\sup_{0<\abs{\gw-\gw_k}\leq \eps_A} \abs{\gw-\gw_k}^\ga \norm{R(i\gw,A)}\leq M_A,
      }
      for all $k\in\Igw$
      and $\norm{R(i\gw,A)}\leq M_A$ whenever $\abs{\gw-\gw_k}>\eps_A$ for all $k\in \Igw$.
    \item[\textup{2.}] We have $\ran(B)\subset \ran( (i\gw_k-A)^\gb)$ and $\ran(C^\ast)\subset \ran( (-i\gw_k-A^\ast)^\gg)$ for every $k\in\Igw$.
  \end{itemize}
\end{assumption}

The Riesz Representation Theorem implies that there exists $\set{b_j}_{j=1}^p\subset X$ and $\set{c_j}_{j=1}^p\subset X$ such that
\eq{
BC = \sum_{j=1}^p \iprod{\cdot}{c_j}b_j.
}
The second part of Assumption~\ref{ass:Astandass} is therefore equivalent to requiring $\set{b_j}_{j=1}^p\subset \ran( (i\gw_k-A)^\gb)$ and $\set{c_j}_{j=1}^p\subset \ran( (i\gw_k-A)^\gg)$ for every $k\in\Igw$.
This immediately implies that $(i\gw_k-A)^{-\gb} B$ and $(-i\gw_k-A^\ast)^{-\gg} C^\ast$ are bounded operators.

 Our first main result concerns the preservation of strong stability.

\begin{theorem}
  \label{thm:stabpert}
  Let Assumption~\textup{\ref{ass:Astandass}} be satisfied with $\ga = \gb+\gg$. There exists $\gd>0$ such that if $\norm{B}<\gd$, $\norm{C}<\gd$ and
  \eq{
    \norm{(i\gw_k-A)^{-\gb} B}<\gd, \qquad  \norm{(-i\gw_k-A^\ast)^{-\gg}C^\ast}<\gd
  }
  for all $k\in \Igw$, then
  the semigroup generated by $A+BC$ is strongly stable. 

  In particular, the spectrum of $A+BC$ satisfies $\gs(A+BC)\cap i\R = \gs_c(A+BC)\cap i\R \subset \set{i\gw_k}_{k\in\Igw}$, and for all $k\in\Igw$
  \eq{
  \sup_{0<\abs{\gw-\gw_k}\leq\eps_A}\abs{\gw-\gw_k}^\ga\norm{R(i\gw,A+BC)} <\infty.
  }
\end{theorem}

The proof of Theorem~\ref{thm:stabpert} is divided into two parts. In Section~\ref{sec:specpert} we study the change of the spectrum of $A$. In Section~\ref{sec:stabpert} we complete the proof of Theorem~\ref{thm:stabpert} by showing that the uniform boundedness of $T(t)$ is preserved under the perturbations.

We remark that the polynomial growth condition for the resolvent was assumed to be satisfied for $\ga\geq 1$. The following lemma shows that this assumption does not result in any loss of generality.

\begin{lemma}
  If $\gs(A)\cap i\R\neq \varnothing$, then $\ga\geq 1$ in the condition~\textup{\eqref{eq:Aresolventgrowthorder}}.
\end{lemma}

\begin{proof}
  Let $k\in \Igw$.
  Since $i\gw_k\in \gs(A)$, for $i\gw$ near $i\gw_k$ we have $\dist(i\gw,\gs(A))\leq \abs{\gw-\gw_k}$. Thus for all such $\gw$ satisfying $0<\abs{\gw-\gw_k}\leq \eps_A$ the standard estimate~\cite[Cor. IV.1.14]{engelnagel} $\norm{R(\gw,A)}\geq \dist(\gl,\gs(A))^{-1}$ implies 
  \eq{
  \frac{1}{\abs{\gw-\gw_k}}\leq \frac{1}{\dist(i\gw,\gs(A))} \leq \norm{R(i\gw,A)} \leq \frac{M_A}{\abs{\gw-\gw_k}^\ga},
  }
  which further implies $\abs{\gw-\gw_k}^{\ga-1} \leq M_A$. However, for small $\abs{\gw-\gw_k}$ this is only possible if $\ga\geq 1$.
\end{proof}

Our second main result concerns the preservation of polynomial stability of a semigroup.
The semigroup $T(t)$ generated by $A$ on the Hilbert space $X$ is called \keyterm{polynomially stable} if $T(t)$ is uniformly bounded, if $\gs(A)\cap i\R=\varnothing$, and if there exists $\ga>0$ and $M\geq 1$ such that 
\eq{
\norm{T(t)A\inv} \leq \frac{M}{t^{1/\ga}}, \qquad \forall t>0.
}
The following theorem gives conditions for the preservation of the polynomial stability under bounded finite rank perturbations. The theorem extends the results in~\cite{Pau12,Pau13a} by allowing the exponents $\gb\geq 0$ and $\gg\geq 0$ to be real numbers.

\begin{theorem}
  \label{thm:Apolpert}
  Assume $T(t)$ generated by $A$ is polynomially stable with exponent $\ga>0$, and $\gb,\gg\geq 0$ are such that $\gb+\gg\geq \ga$.
  There exists $\delta>0$ such that if
  $B\in \Lin(\C^p,X)$ and $C\in \Lin(X,\C^p)$ satisfy
  \eqn{
  \label{eq:BCpolstabrancond}
  \ran(B)\subset \Dom( (-A)^\gb), \quad \mbox{and} \quad \ran(C^\ast)\subset \Dom( (-A^\ast)^\gg)
  }
  and
  $\norm{(-A)^\gb B}<\gd$ and $\norm{(-A^\ast)^{\gg} C^\ast}<\gd$, then the semigroup generated by $A+BC$ is polynomially stable with the same exponent $\ga$.
\end{theorem}

\section{Perturbation of the Spectrum}
\label{sec:specpert}

In this section we show that under the conditions of Theorem~\ref{thm:stabpert} the spectrum of the perturbed operator satisfies $\gs(A+BC)\subset \overline{\C^+}$ and $\gs(A+BC)\cap i\R\subset \set{i\gw_k}_{k\in\Igw}$ are not eigenvalues of $A+BC$.  
On its own, this result is valid 
under weaker assumptions than those in Theorem~\ref{thm:stabpert}.
In particular, we can assume $\gs(A)\cap i\R = \set{i\gw_k}_{k\in\Igw}$ for a countable set $\Igw$ of indices
if the points $\gw_k$ have a finite gap $d_A=\inf_{k\neq l}\abs{\gw_k-\gw_l}>0$. 
Moreover, the perturbation $BC$ does not need to be of finite rank. Instead, for some Banach space $Y$ we can consider $B\in \Lin(Y,X)$ and $C\in \Lin(X,Y)$ such that $(i\gw_k-A)^{-\gb}B$ and $(-i\gw_k-A^\ast)^{-\gg}C^\ast$ are bounded operators for all $k\in\Igw$.

\begin{theorem}
  \label{thm:specpert}
  Let Assumption~\textup{\ref{ass:Astandass}} be satisfied with $\ga=\gb+\gg$. 
  There exists $\gd>0$ such that if $\norm{B}<\gd$, $\norm{C}<\gd$, and
  \eq{
    \norm{(i\gw_k-A)^{-\gb} B}<\gd, \qquad  \norm{(-i\gw_k-A^\ast)^{-\gg}C^\ast}<\gd,
  }
  for all $k\in\Igw$, then
  $\overline{\C^+}\setminus \set{i\gw_k}_{k\in\Igw}\subset \rho(A+BC)$ and
  $i\gw_k\notin \gs_p(A+BC)$ for all $k\in\Igw$.
  In particular, under the above conditions we have
  \eq{
\sup_{\gl\in \overline{\C^+}\setminus \set{i\gw_k}_k} \; \norm{(I-CR(\gl,A)B)\inv}<\infty.
  }
\end{theorem}

We prove the theorem in parts. For the study of the change of the spectrum of $A$ we use the Shermann--Morrison--Woodbury formula given in the following lemma.

\begin{lemma}
  \label{lem:ShermanMorrisonWoodbury}
  Let $\gl\in\rho(A)$, $B  \in \Lin(Y,X)$, $C\in \Lin(X,Y)$. If $1\in \rho(CR(\gl,A)B)$, then $\gl\in\rho(A+BC )$ and
  \eq{
  R(\gl,A+BC) = R(\gl,A)+ R(\gl,A) B(I- CR(\gl,A)B)\inv C R(\gl,A).
  }
\end{lemma}

The Moment Inequality~\cite[Prop. 6.6.4]{haasefuncalc} is one of our most important tools
in dealing with non-integer exponents $\ga$, $\gb$, and $\gg$. The following lemma collects the most frequently used estimates of this type.

\begin{lemma}
  \label{lem:Momentineq}
  Assume $A$ generates a uniformly bounded semigroup and $\gs_p(A)\cap i\R=\varnothing $.
Let $0<\tilde{\ga}<\ga$, $0<\tilde{\gb}<\gb$, and $0<\tilde{\gg}<\gg$. Then there exists 
  $\Ma, \Mb, \Mc\geq 1$ 
  such that for all $\gw\in \R$ we have
  \eq{
  \norm{(i\gw-A)^{\tilde{\ga}}x} 
  &\leq \Ma \norm{x}^{1-\tilde{\ga}/\ga}\norm{(i\gw-A)^{\ga}x}^{\tilde{\ga}/\ga} \qquad \forall x\in \Dom( (i\gw-A)^\ga)\\
  \norm{(i\gw-A)^{-\tilde{\gb}}x}   &\leq 
 \Mb \norm{x}^{1-\tilde{\gb}/\gb}\norm{(i\gw-A)^{-\gb}x}^{\tilde{\gb}/\gb} \qquad \forall x\in \Dom( (i\gw-A)^{-\gb})\\
  \norm{(-i\gw-A^\ast)^{-\tilde{\gg}}x}  &\leq 
 \Mc \norm{x}^{1-\tilde{\gg}/\gg}\norm{(-i\gw-A^\ast)^{-\gg}x}^{\tilde{\gg}/\gg} \qquad \forall x\in \Dom( (-i\gw-A^\ast)^{-\gg}).
  }
  If $B$ and $C$ satisfy the conditions of Assumption~\textup{\ref{ass:Astandass}}, then 
  $(i\gw_k - A)^{-\tilde{\gb}}B$ and $(-i\gw_k-A^\ast)^{-\tilde{\gg}}C^\ast$ are bounded operators and their norms satisfy 
\eq{
\norm{(i\gw_k-A)^{-\tilde{\gb}}B}   &\leq 
\Mb\norm{B}^{1-\tilde{\gb}/\gb}\norm{(i\gw_k-A)^{-\gb}B}^{\tilde{\gb}/\gb} \\
\norm{(-i\gw_k-A^\ast)^{-\tilde{\gg}}C^\ast}  &\leq 
\Mc\norm{C}^{1-\tilde{\gg}/\gg}\norm{(-i\gw_k-A^\ast)^{-\gg}C^\ast}^{\tilde{\gg}/\gg} .
}
\end{lemma}

\begin{proof}
  Let $M\leq 1$ be such that $\norm{T(t)}\leq M$ for all $t\geq 0$. Now $\re\gl\norm{R(\gl,i\gw-A)}\leq M$ for all $\gl\in \C^+$ by the Hille--Yosida Theorem. Using this it is easy to show that the families $(i\gw-A)_{\gw\in\R}$ and $ (-i\gw-A^\ast)_{\gw\in\R}$ of operators are uniformly sectorial of angle $\pi/2$~\cite[Sec. 2.1]{haasefuncalc}. Since $i\gw-A$ is injective, also $((i\gw-A)\inv)_{\gw\in\R}$ and $( (-i\gw-A^\ast)\inv)_{\gw\in\R}$ are uniformly sectorial of angle $\pi/2$ by~\cite[Prop. 2.1.1]{haasefuncalc}.

  For a fixed $\gw\in\R$ the first inequalities now follow from~\cite[Prop. 6.6.4]{haasefuncalc}. However, by~\cite[Prop. 2.6.11]{haasefuncalc} and the uniform sectoriality of the operator families it is possible to choose $\Ma$, $\Mb$, and $\Mc$ to be independent of $\gw\in\R$.

  The boundedness of the operators $(i\gw_k - A)^{-\tilde{\gb}}B$ and $(-i\gw_k-A^\ast)^{-\tilde{\gg}}C^\ast$ and the remaining inequalities follow directly from applying the first inequalities to $x=By\in \Dom( (i\gw_k-A)^{-\gb})$ and $x=C^\ast y\in \Dom( (-i\gw_k - A^\ast)^{-\gg})$, respectively.  \end{proof}

  The last two inequalities in Lemma~\ref{lem:Momentineq} give us a way of estimating the graph norms for exponents $0<\tilde{\gb}<\gb$ and $0<\tilde{\gg}<\gg$. In particular, the following corollary shows that the norm $\norm{(i\gw_k-A)^{-\tilde{\gb}}B}$ can be made arbitrarily small if $\norm{B}$ and $\norm{(i\gw_k-A)^{-\gb}B}$ are small enough, and analogously for the operator $C$.

\begin{corollary}
  \label{cor:normrels}
  Let Assumption~\textup{\ref{ass:Astandass}} be satisfied, and let $0<\tilde{\gb}<\gb$, $0<\tilde{\gg}<\gg$, and $k\in \Igw$. If for some $\gd>0$ the operators $B$ and $C$ satisfy $\norm{B}<\gd$, $\norm{(i\gw_k-A)^{-\gb}B}<\gd$, $\norm{C}<\gd$, and $\norm{(-i\gw_k-A^\ast)^{-\gg}C^\ast}<\gd$, then 
  \eq{
  \norm{(i\gw_k-A)^{-\tilde{\gb}}B}<\Mb\gd, \qquad
  \norm{(-i\gw_k-A^\ast)^{-\tilde{\gg}}C^\ast}<\Mc\gd.
  }
\end{corollary}

We begin the proof of Theorem~\ref{thm:specpert} by showing that we can choose $\gd>0$ in such a way that $\norm{CR(\gl,A)B}\leq c<1$ for all $\gl\in \bigcup_k \Omega_k$, where
$\Omega_k = \setm{\gl\in \C}{\re\gl \geq 0, ~0< \abs{\gl-i\gw_k} \leq \eps_A}$ (see Figure~\ref{fig:Omegak}). 

\begin{figure}[ht]
  \begin{center}
    \includegraphics[width=80pt]{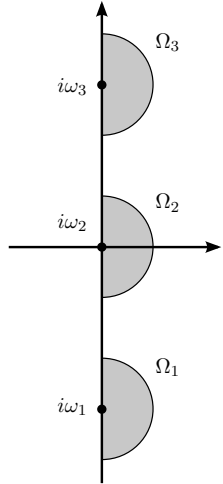}
  \end{center}
  \caption{The domains $\Omega_k$.}
  \label{fig:Omegak}
\end{figure}

\begin{lemma}
  \label{lem:Okbdd}
  If Assumption~\textup{\ref{ass:Astandass}} is satisfied, then there exists $M_0\geq 1$ such that 
  \eq{
  \sup_{\gl\in \Omega_k} \abs{\gl-i\gw_k}^\ga \norm{R(\gl,A)}\leq M_0
  }
  for every $k\in \Igw$.
\end{lemma}

\begin{proof}
  Let $M>0$ be such that $\norm{T(t)}\leq M$. From Assumption~\ref{ass:Astandass} we have 
\eq{
 \sup_{0<\abs{\gw-\gw_k}\leq \eps_A} \abs{\gw-\gw_k}^\ga \norm{R(i\gw,A)} \leq M_A.
}
The Hille--Yosida Theorem~\cite[Thm. II.3.8]{engelnagel} implies that $\re\gl\norm{R(\gl,A)}\leq M$ whenever $\re\gl>0$.

Let $\gl = \mu+i\gw\in \Omega_k$. For $\mu=0$ the bound $\abs{\gl-i\gw_k}^\ga \norm{R(\gl,A)} = \abs{\gw-\gw_k}^\ga \norm{R(i\gw,A)}\leq M_A$ follows directly from~\eqref{eq:Aresolventgrowthorder}. 
On the other hand, if $\gw=\gw_k$ and $\gl=\mu>0$, then the Hille--Yosida Theorem implies
\eq{
\abs{\gl-i\gw_k}^\ga \norm{R(\gl,A)}
=\mu^\ga \norm{R(\gl,A)}
\leq \mu \norm{R(\gl,A)}
\leq M
}
since $\mu^\ga \leq \mu$ due to the fact that $\ga\geq 1$ and $0<\mu\leq \eps_A\leq 1$.
It remains to consider the case $\gl = \mu + i\gw \in \Omega_k$ with $\mu> 0$ and $\gw\neq \gw_k$. In particular, we then have $0<\abs{\gw-\gw_k} \leq \eps_A$ and $0<\mu\leq \eps_A\leq 1$. 
Since $\ga \geq 1 $ and $0<\mu\leq 1$, we have $\mu^\ga \leq \mu$ and
\eq{
\abs{\gl-i\gw_k}^\ga
&= (\mu^2 + (\gw-\gw_k)^2)^{\ga/2} 
\leq (2\max\set{\mu^2 , (\gw-\gw_k)^2})^{\ga/2}
= 2^{\ga/2} \max\set{\mu^\ga , \abs{\gw-\gw_k}^\ga}\\
&\leq 2^{\ga/2} (\mu^\ga + \abs{\gw-\gw_k}^\ga)
\leq 2^{\ga/2} (\mu + \abs{\gw-\gw_k}^\ga),
}
and thus using the resolvent identity $R(\gl,A) = R(i\gw,A) + \mu R(\gl,A)R(i\gw,A)$ we get
\eq{
\MoveEqLeft\abs{\gl-i\gw_k}^\ga \norm{R(\gl,A)} 
\leq 2^{\ga/2} (\mu + \abs{\gw-\gw_k}^\ga) \norm{R(\gl,A)} \\
&= 2^{\ga/2} \mu \norm{R(\gl,A)}+ 2^{\ga/2}\abs{\gw-\gw_k}^\ga \norm{R(i\gw,A) + \mu R(\gl,A)R(i\gw,A)} \\
&\leq 2^{\ga/2} M+ 2^{\ga/2}\abs{\gw-\gw_k}^\ga \norm{R(i\gw,A)} (1+ \mu \norm{R(\gl,A)} )\\
&\leq 2^{\ga/2} \left( M +  M_A (1+ M ) \right).
}
Since in each of the situations the bound for $\abs{\gl-i\gw_k}^\ga \norm{R(\gl,A)}$ is independent of $k\in \Igw$, this concludes the proof.
\end{proof}

\begin{lemma} 
  \label{lem:ARfracbnd}
  Let Assumption~\textup{\ref{ass:Astandass}} be satisfied and denote $\ga = n+\tilde{\ga}$ with $n\in\N$ and $0\leq \tilde{\ga}<1$. There exists $M_1\geq 1$ (not depending on $k\in\Igw$) such that 
  \eq{
  \sup_{\gl\in\Omega_k} \abs{\gl-i\gw_k}^n \norm{(i\gw_k-A)^{\tilde{\ga}}R(\gl,A)}\leq M_1
  }
  for all $k\in\Igw$.
\end{lemma}

\begin{proof}
  By Lemma~\ref{lem:Okbdd} there exists $M_0\geq 1$ such that $\abs{\gl-i\gw_k}^\ga \norm{R(\gl,A)}\leq M_0$ for all $k\in\Igw$.  
Let $k\in\Igw$, $\gl\in \Omega_k$, and denote $R_\gl = R(\gl,A)$, $A_k = A-i\gw_k$, and $\gl_k=\gl-i\gw_k$ for brevity.

If $\ga=n$ and $\tilde{\ga}=0$, we have 
\eq{
\abs{\gl_k}^n \norm{(-A_k)^{\tilde{\ga}}R_\gl}
=\abs{\gl_k}^\ga \norm{R_\gl}\leq M_0.
}
Thus the claim is satisfied with $M_1=M_0$, which is independent of $k\in\Igw$.

If $0<\tilde{\ga}<1$, then 
by Lemma~\ref{lem:Momentineq} there exists a constant $M_{\tilde{\ga}}$ independent of $k\in\Igw$ such that $\norm{(-A_k)^{\tilde{\ga}}x}\leq M_{\tilde{\ga}} \norm{x}^{1-\tilde{\ga}}\norm{(-A_k)x}^{\tilde{\ga}}$ for all $x\in \Dom( A)$. This further implies $\norm{(-A_k)^{\tilde{\ga}}R_\gl}\leq M_{\tilde{\ga}} \norm{R_\gl}^{1-\tilde{\ga}}\norm{(-A_k)R_\gl}^{\tilde{\ga}}$. Using
\eq{
(-A_k)R_\gl=(i\gw_k-A)R_\gl = (i\gw_k-\gl +\gl - A )R_\gl = (i\gw_k-\gl) R_\gl + I = -\gl_k R_\gl + I
}
and the scalar inequality $(a+b)^{\tilde{\ga}} \leq 2^{\tilde{\ga}}(a^{\tilde{\ga}}+b^{\tilde{\ga}})$
we get 
\eq{
\MoveEqLeft \abs{\gl_k}^n  \norm{(-A_k)^{\tilde{\ga}} R_\gl} 
\leq M_{\tilde{\ga}}  \abs{\gl_k}^n \norm{R_\gl}^{1-\tilde{\ga}} \norm{(-A_k) R_\gl}^{\tilde{\ga}}
\leq M_{\tilde{\ga}} \abs{\gl_k}^n \norm{R_\gl}^{1-\tilde{\ga}} (1+\abs{\gl_k}\norm{R_\gl})^{\tilde{\ga}}\\
&\leq 2^{\tilde{\ga}}M_{\tilde{\ga}}  \abs{\gl_k}^n \norm{R_\gl}^{1-\tilde{\ga}} (1+\abs{\gl_k}^{\tilde{\ga}}\norm{R_\gl}^{\tilde{\ga}})
\leq 2^{\tilde{\ga}}M_{\tilde{\ga}} \left[ (\abs{\gl_k}^{\frac{n}{1-\tilde{\ga}}} \norm{R_\gl})^{1-\tilde{\ga}} +\abs{\gl_k}^{n+\tilde{\ga}}\norm{R_\gl} \right] .
}
Since $n= \floor{\ga}\geq 1$ we have 
\eq{
\frac{n}{1-\tilde{\ga}}
=\frac{n(n+\tilde{\ga})}{(1-\tilde{\ga})(n+\tilde{\ga})}
=\frac{n(n+\tilde{\ga})}{n-\tilde{\ga}(n-1) -\tilde{\ga}^2}
\geq\frac{n(n+\tilde{\ga})}{n}
=n+\tilde{\ga}
= \ga.
}
Since $\gl\in \Omega_k$, we have $\abs{\gl_k}\leq\eps_A\leq 1$, and thus $\abs{\gl_k}^{\frac{n}{1-\tilde{\ga}}}\leq \abs{\gl_k}^{\ga}$, and
\eq{
\MoveEqLeft \abs{\gl_k}^n  \norm{(-A_k)^{\tilde{\ga}} R_\gl} 
\leq 2^{\tilde{\ga}}M_{\tilde{\ga}} \left[ (\abs{\gl_k}^{\frac{n}{1-\tilde{\ga}}} \norm{R_\gl})^{1-\tilde{\ga}} +\abs{\gl_k}^{n+\tilde{\ga}}\norm{R_\gl} \right] \\
&\leq 2^{\tilde{\ga}}M_{\tilde{\ga}}  \left[ (\abs{\gl_k}^\ga \norm{R_\gl})^{1-\tilde{\ga}} +\abs{\gl_k}^\ga\norm{R_\gl} \right]
\leq 2^{\tilde{\ga}}M_{\tilde{\ga}}  \left[ M_0^{1-\tilde{\ga}} +M_0 \right]
\leq 2^{\tilde{\ga}+1}M_{\tilde{\ga}} M_0  
}
since it was assumed that $M_0\geq 1$. Therefore the claim holds with $M_1=2^{\tilde{\ga}+1}M_{\tilde{\ga}}M_0$, which is independent of $k\in\Igw$.  
\end{proof}

\begin{lemma}
  \label{lem:ARboundzero}
Let Assumption~\textup{\ref{ass:Astandass}} be satisfied with $\ga = \gb+\gg$ and let $0<c<1$. There exists $\gd>0$ such that if $\norm{B}<\gd$, $\norm{C}<\gd$,
  \eq{
  \norm{(i\gw_k-A)^{-\gb} B}<\gd, \quad \mbox{and} \quad \norm{(i\gw_k-A^\ast)^{-\gg}C^\ast}<\gd,
  }
  for all $k\in \Igw$,
  then $\norm{CR(\gl,A)B}\leq c<1$ for all $\gl\in \bigcup_{k\in\Igw}\Omega_k$.
\end{lemma}

\begin{proof}
  By Lemma~\ref{lem:Okbdd} there exists $M_0\geq 1$ such that $\abs{\gl-i\gw_k}^\ga \norm{R(\gl,A)}\leq M_0$ for all $k\in \Igw$ and $\gl\in\Omega_k$. Let $k\in \Igw$ be fixed.

  Choose $m,n\in \N_0$ and $\tilde{\ga},\tilde{\gb},\tilde{\gg}\in [0,1)$ such that $\ga= \floor{\ga}+\tilde{\ga}$, $\gb = m+ \tilde{\gb}$ and $\gg = n+\tilde{\gg}$. Since $\gb + \gg=\ga$, we have either $m+n=\floor{\ga}$ and $\tilde{\gb}+\tilde{\gg}=\tilde{\ga}$, or alternatively, $m+n+1=\floor{\ga}$ and $\tilde{\gb}+\tilde{\gg}=\tilde{\ga}+1$.
 
  For $0<r\leq \gb$ denote $B_r = (i\gw_k-A)^{-r}B$, and for $0<r\leq \gg$ denote $\tilde{C}_r= (-i\gw_k-A^\ast)^{-r}C^\ast$.

In order to shows the existence of an appropriate bound $\gd>0$, we begin by showing that the norms $\norm{CR(\gl,A)B}$ for $\gl\in\Omega_k$ can be estimated using the norms $\norm{B}$, $\norm{C}$, $\norm{(i\gw_k-A)^{-\gb}B}$ and $\norm{(-i\gw_k-A^\ast)^{-\gg}C^\ast}$.

Let $x,y\in \Y$ be such that $\norm{x}=\norm{y}=1$.
  For brevity, denote $R_\gl = R(\gl,A)$, $A_k = A-i\gw_k$, and $\gl_k=\gl-i\gw_k$. We have $(-A_k)R_\gl=(i\gw_k-A)R_\gl = (i\gw_k-\gl +\gl - A )R_\gl = (i\gw_k-\gl) R_\gl + I = -\gl_k R_\gl + I$. For all $\gl\in \Omega_k$ we have 
\eq{
\MoveEqLeft[1]\abs{\iprod{CR_\gl B x}{y}}
=\abs{\iprod{C(-A_k)R_\gl B_1x}{y}}
=\abs{\iprod{C( -\gl_k R_\gl + I )B_1x}{y}}\\
&\leq \abs{\gl_k}\abs{\iprod{ CR_\gl B_1x}{y}} + \norm{CB_1} 
\leq \abs{\gl_k}\abs{ \iprod{C(-A_k)R_\gl B_2x}{y}} + \norm{CB_1} \\
&\leq \abs{\gl_k}^2\abs{ \iprod{CR_\gl B_2 x}{y}} +\abs{\gl_k}\norm{CB_2} +  \norm{CB_1}\\
&\leq \dots \leq \abs{\gl_k}^m \abs{\iprod{CR_\gl B_m x}{y}} +  \sum_{l=1}^m \abs{\gl_k}^{l-1}\norm{CB_l}\\
&\leq \abs{\gl_k}^m \abs{\iprod{CR_\gl B_m x}{y}} +  \sum_{l=1}^m \norm{B_l}\norm{C}
}
since $\abs{\gl_k}\leq \eps_A\leq 1$.
We can further estimate the term $\abs{\iprod{CR_\gl B_m x}{y}}$ by 
\eq{
\MoveEqLeft
\abs{\iprod{CR_\gl B_m x}{y}}
=\abs{\iprod{R_\gl B_m x}{C^\ast y}}
=\abs{\iprod{(-A_k)R_\gl B_m x}{(-A_k^\ast)\inv C^\ast y}}\\
&=\abs{\iprod{(-\gl_kR_\gl+I) B_m x}{\tilde{C}_1 y}}
\leq \abs{\gl_k} \abs{\iprod{R_\gl B_m x}{\tilde{C}_1 y}} + \norm{B_m}\norm{\tilde{C}_1 } \\
&\leq \abs{\gl_k}^2 \abs{\iprod{R_\gl B_m x}{\tilde{C}_2 y}} + \abs{\gl_k}\norm{B_m}\norm{\tilde{C}_2} + \norm{B_m}\norm{\tilde{C}_1 } \\
&\leq \cdots\leq \abs{\gl_k}^n \abs{\iprod{R_\gl B_m x}{\tilde{C}_my}} + \sum_{l=1}^n \abs{\gl_k}^{l-1} \norm{B_m}\norm{\tilde{C}_l} \\
&\leq \abs{\gl_k}^n \abs{\iprod{R_\gl B_m x}{\tilde{C}_ny}} + \sum_{l=1}^n \norm{B_n}\norm{\tilde{C}_l}  .
}
Combining these estimates we get
\eqn{
\label{eq:Okbddsumest}
\MoveEqLeft\abs{\iprod{CR_\gl Bx}{y}} 
\leq  \abs{\gl_k}^{m+n} \abs{\iprod{R_\gl B_mx}{\tilde{C}_n y}} + \sum_{l=1}^n \norm{B_m}\norm{\tilde{C}_l} +  \sum_{l=1}^m \norm{B_l}\norm{C}.
}
We have from Lemma~\ref{lem:Momentineq} that for all $l$ the norms $\norm{B_l}$ and $\norm{\tilde{C}_l}$ can estimated using the norms $\norm{B}$, $\norm{C}$, $\norm{(i\gw_k-A)^{-\gb}B}$, and $\norm{(-i\gw_k-A^\ast)^{-\gg}C^\ast}$, and these estimates do not depend on $k\in \Igw$.
However, we need to consider the term $\abs{\gl_k}^{m+n} \abs{\iprod{R_\gl B_mx}{\tilde{C}_n y}}$ separately. 

If $\tilde{\gb}+\tilde{\gg}=\tilde{\ga}<1$, then we have from Lemma~\ref{lem:ARfracbnd} that there exists $M_1\geq 1$ (independent of $k\in\Igw$) such that $\abs{\gl_k}^{n+m} \norm{(-A_k)^{\tilde{\ga}}R_\gl}\leq M_1$ for all $\gl\in\Omega_k$, and thus
\eq{
\MoveEqLeft \abs{\gl_k}^{m+n} \abs{\iprod{R_\gl B_mx}{\tilde{C}_ny}}
=\abs{\gl_k}^{m+n} \abs{ \iprod{(-A_k)^{\tilde{\gb}+\tilde{\gg}} R_\gl B_\gb x }{\tilde{C}_\gg y}}\\
&\leq \abs{\gl_k}^{m+n}  \norm{(-A_k)^{\tilde{\ga}} R_\gl} \norm{B_\gb}\norm{\tilde{C}_\gg}
\leq M_1 \norm{B_\gb}\norm{\tilde{C}_\gg}.
}

Alternatively, if $\tilde{\gb}+\tilde{\gg}=\tilde{\ga}+1\geq 1$, then $\floor{\ga}=m+n+1$ and we necessarily have $0<\tilde{\gb},\tilde{\gg}<1$. We can choose $\gb_1 = \tilde{\gb}/(\tilde{\ga}+1)$ and $\gg_1 = \tilde{\gg}/(\tilde{\ga}+1)$, which satisfy $0<\gb_1\leq \tilde{\gb}$, $0<\gg_1\leq \tilde{\gg}$, and $\gb_1+\gg_1=1$. Now we can use $(\tilde{\gb}-\gb_1) + (\tilde{\gg}-\gg_1) = \tilde{\gb} + \tilde{\gg} -1 = \tilde{\ga}<1$ and estimate
\eq{
\MoveEqLeft \abs{\gl_k}^{m+n} \abs{\iprod{R_\gl B_mx}{\tilde{C}_ny}}
=\abs{\gl_k}^{m+n} \abs{\iprod{(-A_k) R_\gl B_{m+\gb_1} x}{\tilde{C}_{n+\gg_1} y}}\\
&\leq \abs{\gl_k}^{\floor{\ga}} \abs{\iprod{R_\gl B_{m+\gb_1}x}{\tilde{C}_{n+\gg_1} y}} + \abs{\gl_k}^{m+n} \norm{B_{m+\gb_1}}\norm{\tilde{C}_{n+\gg_1}}\\
&\leq \abs{\gl_k}^{\floor{\ga}} \abs{\iprod{(-A_k)^{\tilde{\gg}-\gg_1}R_\gl (-A_k)^{\tilde{\gb}-\gb_1} B_\gb x}{\tilde{C}_\gg y}} + \norm{B_{m+\gb_1}}\norm{\tilde{C}_{n+\gg_1}}\\
&= \abs{\gl_k}^{\floor{\ga}} \abs{\iprod{(-A_k)^{\tilde{\ga}} R_\gl B_\gb x}{\tilde{C}_\gg y}} + \norm{B_{m+\gb_1}}\norm{\tilde{C}_{n+\gg_1}}\\
&\leq \abs{\gl_k}^{\floor{\ga}} \norm{(-A_k)^{\tilde{\ga}} R_\gl} \norm{B_\gb }\norm{\tilde{C}_\gg } + \norm{B_{m+\gb_1}}\norm{\tilde{C}_{n+\gg_1}}\\
&\leq M_1 \norm{B_\gb }\norm{\tilde{C}_\gg } + \norm{B_{m+\gb_1}}\norm{\tilde{C}_{n+\gg_1}}.
}

We have
\eq{
\norm{CR(\gl,A)B} = \sup_{\norm{x}=\norm{y}=1} \abs{\iprod{CR(\gl,A)Bx}{y}},
}
and thus~\eqref{eq:Okbddsumest} together with the subsequent estimates shows that for $\gl\in \Omega_k$ the norm $\norm{CR(\gl,A)B}$ can be estimated independently of $\gl$. Moreover, it is clear that the bound can be made arbitrarily small (in particular, to be smaller than $c<1$) if 
$\norm{B}$, $\norm{C}$, $\norm{B_j}$, and $\norm{\tilde{C}_l}$ for all $j\in \List{m}$, $l\in \List{n}$ (plus $\norm{B_{m+\gb_1}},\norm{\tilde{C}_{n+\gg_1}}$ if $\tilde{\gb}+\tilde{\gg}>1$) are small enough. However, by Corollary~\ref{cor:normrels} we can see that each of these norms for $0<\gb_0<\gb$ and $0<\gg_0<\gg$ can be estimated as
\eq{
\norm{B_{\gb_0}} \leq M_{\gb_0/\gb} \norm{B}^{1-\gb_0/\gb} \norm{B_\gb}^{\gb_0/\gb}, 
\quad \mbox{and} \quad
\norm{\tilde{C}_{\gg_0}} \leq M_{\gg_0/\gg} \norm{B}^{1-\gg_0/\gg} \norm{\tilde{C}_\gg}^{\gg_0/\gg}, 
}
where the constants $M_{\gb_0/\gb}$ and $M_{\gg_0/\gg}$ only depend on the exponents, and not on $k\in\Igw$. This finally implies that for $0<c<1$ there exists $\gd>0$ in such a way that
\eq{
\sup_{\gl\in\Omega_k} \norm{CR(\gl,A)B}\leq c
}
whenever $\norm{B}<\gd$, $\norm{C}<\gd$, $\norm{(i\gw_k-A)^{-\gb}B}<\gd$, and $\norm{(-i\gw_k-A^\ast)^{-\gg}C^\ast}<\gd$.

Due to the fact that none of the used estimates depend on $k\in\Igw$, the same bound $\gd>0$ works for all indices $k\in\Igw$. This concludes the proof.
\end{proof}

\begin{lemma}
  \label{lem:RbddOcomp}
Let Assumption~\textup{\ref{ass:Astandass}} be satisfied. There exists $M_2\geq 1$ such that
\eq{
\sup_{\gl\in \overline{\C^+}\setminus (\cup_k \Omega_k  )} \norm{R(\gl,A)}\leq M_2.
}
\end{lemma}

\begin{proof}
Let $\gl \in \overline{\C^+}\setminus \left( \bigcup_k \Omega_k \right)$ and let $\gl_0$ be such that $0\leq\re\gl_0\leq \re\gl$, $\im\gl_0=\im\gl$ and $\gl_0$ lies on the boundary of $\overline{\C^+}\setminus \left( \bigcup_k \Omega_k \right)$. Then either $\gl_0\in i\R$, which implies $\norm{R(\gl_0,A)}\leq M_A$ by Assumption~\ref{ass:Astandass}, or otherwise $\gl_0\in \Omega_k$ and $\abs{\gl_0-i\gw_k}=\eps_A$ for some $k\in \Igw$. By Lemma~\ref{lem:Okbdd} we have that there exists $M_0$ (independent of $k$) such that in this case we have
\eq{
\abs{\gl_0-i\gw_k}^\ga \norm{R(\gl_0,A)}\leq M_0 
\qquad \Leftrightarrow \qquad 
\norm{R(\gl_0,A)}\leq \frac{M_0}{\eps_A^\ga} .
}
Finally, if $M\geq 1$ is such that $\norm{T(t)}\leq M$, then $\re\gl \norm{R(\gl,A)}\leq M$ by the Hille--Yosida Theorem.
Using the resolvent identity $R(\gl,A) = R(\gl_0,A) + (\gl-\gl_0) R(\gl_0,A)R(\gl,A)$ we get
\eq{
\MoveEqLeft \norm{R(\gl,A)} 
\leq \norm{R(\gl_0,A)}(1 + \abs{\gl-\gl_0} \norm{R(\gl,A)})\\
&\leq \max\set{M_A,M_0/\eps_A^\ga} (1+(\re\gl-\re\gl_0)\norm{R(\gl,A)})\\
&\leq \max\set{M_A,M_0/\eps_A^\ga} (1+\re\gl\norm{R(\gl,A)})
\leq \max\set{M_A,M_0/\eps_A^\ga} (1+M)
=:M_2.
}
\end{proof}

\begin{lemma}
  \label{lem:ABCinj}
  Let Assumption~\textup{\ref{ass:Astandass}} be satisfied with $\gb+ \gg = \ga$.
  There exists $\gd>0$ such that if $\norm{B}<\gd$, $\norm{C}<\gd$, and
  \eq{
    \norm{(i\gw_k-A)^{-\gb} B}<\gd, \qquad  \norm{(-i\gw_k-A^\ast)^{-\gg}C^\ast}<\gd,
  }
  for all $k\in\Z$, then
  $i\gw_k\notin \gs_p(A + BC)$.
\end{lemma}

\begin{proof}
  Choose $0\leq\gb_1<\gb$ and $0\leq \gg_1<\gg$ in such a way that $\gb_1+\gg_1=1$. Assume $\norm{(i\gw_k-A)^{-\gb_1}B}<1$ and $\norm{(-i\gw_k-A^\ast)^{-\gg_1}C^\ast}<1$. Since $0\leq \gg_1\leq 1$, we have $\ran( i\gw_k-A)\subset \ran( (i\gw_k-A)^{\gg_1}) \subset X$, which implies $\overline{\Dom( (i\gw_k-A)^{-\gg_1})} = X$ due to the fact that $i\gw_k\in \gs_c(A)$. Because of this, the operator $C(i\gw_k-A)^{-\gg_1}$ has a unique bounded extension $C_{\gg_1}\in \Lin(X,\Y)$ with norm $\norm{C_{\gg_1}}=\norm{(-i\gw_k-A^\ast)^{-\gg_1}C^\ast}<1$. 
  
  Let $\phi\in \ker(i\gw_k - A-BC)$.
  Because we have $\norm{(i\gw_k-A)^{-\gb_1}B C_{\gg_1}}<1$, the operator $I - (i\gw_k-A)^{-\gb_1}BC_{\gg_1}$ is boundedly invertible, and since $(i\gw_k-A)^{\gb_1}$ and $(i\gw_k-A)^{\gg_1}$ are injective, we have
  \eq{
  &(i\gw_k - A - BC)\phi = 0\\
  \Rightarrow \quad & 
  (i\gw_k - A)^{\gb_1} (I - (i\gw_k-A)^{-\gb_1}BC(i\gw_k-A)^{-\gg_1})(i\gw_k-A)^{\gg_1}\phi = 0\\
  \Rightarrow \quad & 
  (I - (i\gw_k-A)^{-\gb_1}BC_{\gg_1})(i\gw_k-A)^{\gg_1}\phi = 0\\
  \Rightarrow \quad & 
  (i\gw_k-A)^{\gg_1}\phi = 0
  \quad \Rightarrow \quad 
  \phi = 0.
  }
  Since $\phi \in \ker(i\gw_k -A-BC)$ was arbitrary, this concludes that $i\gw_k \notin \gs_p(A+BC)$.

  Finally, Lemma~\ref{lem:Momentineq} can be used to conclude that there exists $\gd>0$ such that the condition $\norm{(i\gw_k-A)^{-\gb_1}B}<1$ and $\norm{(-i\gw_k-A^\ast)^{-\gg_1}C^\ast}<1$ is satisfied for all $k\in\Igw$ whenever $\norm{B}<\gd$, $\norm{C}<\gd$, $\norm{(i\gw_k-A)^{-\gb} B}<\gd$, and $\norm{(-i\gw_k-A^\ast)^{-\gg}C^\ast}<\gd$ for all $k\in\Igw$.
\end{proof}

\begin{proof}[Proof of Theorem~\textup{\ref{thm:specpert}}]
Let $0<c<1$ and let $M_2\geq 1$ be as in Lemma~\ref{lem:RbddOcomp}.  
Choose $\gd_1>0$ as in Lemma~\ref{lem:ARboundzero}, and $\gd_2>0$ as in Lemma~\ref{lem:ABCinj}. We will show that the claims of the theorem are satisfied with the choice $\gd = \min \set{\gd_1,\gd_2,\sqrt{c/M_2}}$.
To this end, for the rest of the proof, we assume that the operators $B$ and $C$ satisfy $\norm{B}<\gd$, $\norm{C}<\gd$, $\norm{(i\gw_k-A)^{-\gb} B}<\gd$, and $\norm{(-i\gw_k-A^\ast)^{-\gg}C^\ast}<\gd$ for all $k\in\Igw$.

Since $\norm{B},\norm{C}<\gd\leq\sqrt{c/M_2}$, for all $\gl \in \overline{\C^+}\setminus \left( \bigcup_k \Omega_k \right)$ we have
\eq{
\norm{CR(\gl,A)B}\leq \norm{C}\norm{B}\norm{R(\gl,A)} 
<\frac{\sqrt{c}}{\sqrt{M_2}}\cdot\frac{\sqrt{c}}{\sqrt{M_2}}\cdot M_2
=c<1 .
}
Furthermore, since $\gd\leq\gd_1$, we have from Lemma~\ref{lem:ARboundzero} that 
$\norm{CR(\gl,A)B}\leq c<1$ also for $\gl\in \bigcup_k \Omega_k$. Combining these estimates, we can see that 
 $\norm{CR(\gl,A)B}\leq c<1$ and $1\in \rho(CR(\gl,A)B)$ for all $\gl\in \overline{\C^+}\setminus \set{i\gw_k}_{k\in\Igw}$.
The Shermann--Morrison--Woodbury formula in Lemma~\ref{lem:ShermanMorrisonWoodbury} therefore implies that $\overline{\C^+}\setminus \set{i\gw_k}_{k\in\Igw}\subset \rho(A+BC)$. 
Finally, since $\gd\leq \gd_2$, we have from Lemma~\ref{lem:ABCinj} that $i\gw_k\notin \gs_p(A+BC)$ for all $k\in\Igw$.

If $\gl\in \overline{\C^+}\setminus \set{i\gw_k}_k$, then $\norm{CR(\gl,A)B}\leq c<1$ implies
\eq{
\norm{(I-CR(\gl,A)B)\inv} 
= \Norm{\sum_{k=0}^\infty (CR(\gl,A)B)^n}
\leq \sum_{k=0}^\infty \norm{CR(\gl,A)B}^n
\leq \sum_{k=0}^\infty c^n
= \frac{1}{1-c},
}
which concludes the final claim of the lemma. 
\end{proof}

\section{Preservation of Strong Stability}
\label{sec:stabpert}

In this section we complete the proof of Theorem~\ref{thm:stabpert}. In particular, this requires showing that under the stated conditions the perturbed semigroup is uniformly bounded.
For this we use the following condition using the resolvent operators. The proof of the theorem can be found in~\cite[Thm. 2]{gomilkounifbdd}.

\begin{theorem}
  \label{thm:unifbddconds}
  Let $A$ generate a semigroup $T_A(t)$ on a Hilbert space $X$ and let $\gs(A)\subset \overline{\C^-}$. The semigroup $T_A(t)$ is uniformly bounded if and only if 
      for all $x,y\in X$ we have
      \eq{
      \sup_{\xi>0}\, \xi \int_{-\infty}^\infty \norm{R(\xi+i\eta,A)x}^2+ \norm{R(\xi+i\eta,A)^\ast y}^2 d\eta <\infty.
      } 
\end{theorem}

We begin by proving two auxiliary lemmata used in proving the uniform boundedness of the perturbed semigroup, as well as in showing the polynomial growth of the perturbed resolvent operator near the points $i\gw_k$. 

\begin{lemma}
  \label{lem:BCfinrankint}
  If $\tilde{B}\in \Lin(\C^p,X)$ and $\tilde{C}\in \Lin(X,\C^p)$, then 
  \eq{
  &\sup_{\xi>0} \; \xi\int_{-\infty}^\infty \norm{R(\xi+i\eta,A)\tilde{B}}^2 d\eta <\infty,
  \qquad
  &\sup_{\xi>0} \; \xi\int_{-\infty}^\infty \norm{\tilde{C} R(\xi+i\eta,A)}^2 d\eta<\infty.
  } 
\end{lemma}

\begin{proof}
  Let $\set{b_j}_{j=1}^p\subset X$ and $\set{c_j}_{j=1}^p\subset X$ be such that $\tilde{B}u= \sum_{j=1}^p u_jb_j $ for $u\in \C^p$ and $\tilde{C}=\left( \iprod{\cdot}{c_1}, \ldots, \iprod{\cdot}{c_p} \right)^T$. A straightforward estimate can be used to show that (see~\cite[Lem. 3]{Pau12} for the proof). 
  \eq{
  \norm{R(\gl,A) \tilde{B}}^2 
  \leq  \sum_{j=1}^p\, \norm{R(\gl,A) b_j}^2 ,
  \qquad
  \norm{\tilde{C} R(\gl,A) }^2 
  \leq \sum_{j=1}^p \norm{R(\gl,A)^\ast c_j}^2.
  }
  Together with Theorem~\ref{thm:unifbddconds} these estimate conclude
    \eq{
    &\sup_{\xi>0} \; \xi\int_{-\infty}^\infty \norm{R(\xi+i\eta,A)\tilde{B}}^2 d\eta
    \leq \sum_{j=1}^p \; \sup_{\xi>0} \; \xi\int_{-\infty}^\infty \norm{R(\xi+i\eta,A)b_j}^2 d\eta <\infty\\
    &\sup_{\xi>0} \; \xi\int_{-\infty}^\infty \norm{\tilde{C}R(\xi+i\eta,A)}^2 d\eta
    \leq \sum_{j=1}^p \; \sup_{\xi>0} \; \xi\int_{-\infty}^\infty \norm{R(\xi+i\eta,A)^\ast c_j}^2 d\eta <\infty.
    }
\end{proof}

The following lemma contains the most technically demanding estimates used in the proof of Theorem~\ref{thm:stabpert}.

\begin{lemma} 
  \label{lem:RBCRest}
  Let $\gd>0$ be chosen as in Theorem~\textup{\ref{thm:specpert}} and let $k\in\Igw$.
  There exists a function $f_k: \overline{\C^+}\setminus \set{i\gw_l}_{l\in\Igw}\rightarrow \R^+$ such that if $\norm{B}<\gd$, $\norm{C}<\gd$, $\norm{(i\gw_k-A)^{-\gb}B}<\gd$, and $\norm{(-i\gw_k-A^\ast)^{-\gg}C^\ast}<\gd$ we have 
  \eq{
  \norm{R(\gl,A)B}\norm{CR(\gl,A)}\leq f_k(\gl) \qquad \forall \gl\in\Omega_k,
  }
   and $f_k(\cdot)$ has properties $\sup_{0<\abs{\gw-\gw_k}\leq \eps_A}\abs{\gw-\gw_k}^\ga f_k(i\gw)<\infty$ and
  \eq{
  \sup_{\xi>0} \xi \int_{-\infty}^\infty f_k(\xi+i\eta)^2 d\eta<\infty.
  }
\end{lemma}

\begin{proof} 
Choose $m,n\in \N_0$ and $\tilde{\ga},\tilde{\gb},\tilde{\gg}\in [0,1)$ such that $\ga= \floor{\ga}+\tilde{\ga}$, $\gb = m+ \tilde{\gb}$ and $\gg = n+\tilde{\gg}$. Since $\gb + \gg=\ga$, we have either $m+n=\floor{\ga}$ and $\tilde{\gb}+\tilde{\gg}=\tilde{\ga}$, or alternatively, $m+n+1=\floor{\ga}$ and $\tilde{\gb}+\tilde{\gg}-1=\tilde{\ga}$.

  Let $k\in\Igw$ and $\gl\in\Omega_k$. 
  For brevity, denote $R_\gl = R(\gl,A)$, $A_k = A-i\gw_k$, and $\gl_k=\gl-i\gw_k$. 
  Moreover, for $0<r\leq \gb$ denote $B_r = (i\gw_k-A)^{-r}B$, and for $0<r\leq \gg$ denote $\tilde{C}_r= (-i\gw_k-A^\ast)^{-r}C^\ast$.

  Repeatedly applying $(-A_k)R_\gl= (i\gw_k-\gl +\gl - A )R_\gl = (i\gw_k-\gl) R_\gl + I = -\gl_k R_\gl + I$ and using $\abs{\gl_k}=\abs{\gl-i\gw_k}\leq \eps_A\leq 1$ and Corollary~\ref{cor:normrels} we can see that
  \eq{
  \MoveEqLeft\norm{R_\gl B}
  = \norm{(-A_k)R_\gl B_1}
  \leq \abs{\gl_k}\norm{R_\gl B_1} + \norm{B_1}
  = \abs{\gl_k}\norm{(-A_k)R_\gl B_2} + \norm{B_1}\\
  &\leq \abs{\gl_k}^2\norm{R_\gl B_2} + \abs{\gl_k} \norm{B_2} + \norm{B_1}
  \leq \cdots 
  \leq \abs{\gl_k}^m \norm{R_\gl B_m} + \sum_{l=1}^m \abs{\gl_k}^{l-1} \norm{B_l}\\
  &\leq \abs{\gl_k}^m \norm{R_\gl B_m} + \sum_{l=1}^m M_{l/\gb}\norm{B}^{1-l/\gb}\norm{B_\gb}^{l/\gb}
  \leq \abs{\gl_k}^m \norm{R_\gl B_m} + \gd\sum_{l=1}^m M_{l/\gb},
  }
  since $\norm{B_l}\leq M_{l/\gb}\norm{B}^{1-l/\gb} \norm{B_\gb}^{l/\gb}$, and $\norm{B},\norm{B_\gb}<\gd$.

  Similarly using $(-A_k^\ast)R_\gl^\ast = - \conj{\gl_k}R_\gl^\ast + I$ we can estimate
  \eq{
  \MoveEqLeft\norm{CR_\gl}
  = \norm{R_\gl^\ast C^\ast}
  = \norm{(-A_k^\ast)R_\gl^\ast \tilde{C}_1}
  \leq \abs{\gl_k}\norm{R_\gl^\ast \tilde{C}_1} + \norm{\tilde{C}_1}\\
&\leq \abs{\gl_k}^2\norm{R_\gl^\ast \tilde{C}_2} + \abs{\gl_k} \norm{\tilde{C}_2} + \norm{\tilde{C}_1}
  \leq \cdots 
  \leq \abs{\gl_k}^n \norm{R_\gl^\ast \tilde{C}_n} + \sum_{l=1}^n \norm{\tilde{C}_l}\\
  &\leq \abs{\gl_k}^n \norm{R_\gl^\ast \tilde{C}_n} + \gd\sum_{l=1}^n M_{l/\gg}.
  } 
  We thus have 
  \eq{
  \MoveEqLeft\norm{R_\gl B} \norm{CR_\gl}
  \leq \left( \abs{\gl_k}^m \norm{R_\gl B_m} + \gd\sum_{l=1}^m M_{l/\gb} \right)
\norm{CR_\gl}\\
&\leq  \abs{\gl_k}^{m} \norm{R_\gl B_m} \norm{CR_\gl} + \gd\norm{CR_\gl}\sum_{l=1}^m M_{l/\gb}  \\
&\leq  \abs{\gl_k}^{m} \norm{R_\gl B_m} \left( \abs{\gl_k}^n \norm{R_\gl^\ast \tilde{C}_n} + \gd\sum_{l=1}^n M_{l/\gg} \right) + \gd\norm{CR_\gl}\sum_{l=1}^m M_{l/\gb}\\
&\leq \abs{\gl_k}^{m+n} \norm{R_\gl B_m}\norm{R_\gl^\ast \tilde{C}_n} +  \gd \norm{R_\gl B_m} \sum_{l=1}^m  M_{l/\gg}
  + \gd \norm{CR_\gl}\sum_{l=1}^m M_{l/\gb}  
}
If we denote $f_k^0(\gl)=  \gd \norm{R_\gl B_m} \sum_{l=1}^m  M_{l/\gg} + \gd\norm{CR_\gl}\sum_{l=1}^m M_{l/\gb}$, then $\sup_{\xi>0}\; \xi \int_{-\infty}^\infty f_k^0(\xi+i\eta)^2 d\eta<\infty$ by the scalar inequality $(a+b)^2\leq 2(a^2+b^2)$ and Lemma~\ref{lem:BCfinrankint}. Moreover, for any $\gw\in\R$ with $0<\abs{\gw-\gw_k}\leq \eps_A$
\eq{
\MoveEqLeft \abs{\gw-\gw_k}^\ga f_k^0(i\gw)
\leq \abs{\gw-\gw_k}^\ga \norm{R(i\gw,A)}  \left( \gd \norm{ B_m} \sum_{l=1}^m  M_{l/\gg} + \gd\norm{C}\sum_{l=1}^m M_{l/\gb} \right)\\
&\leq M_A  \gd \left( \norm{ B_m} \sum_{l=1}^m  M_{l/\gg} + \norm{C}\sum_{l=1}^m M_{l/\gb} \right)
}
by Assumption~\ref{ass:Astandass}.
Thus if we can find $f_k^1(\cdot)$ in such a way that $ \abs{\gl_k}^{m+n} \norm{R_\gl B_m}\norm{R_\gl^\ast \tilde{C}_n}\leq f_k^1(\gl)$ for all $\gl\in\Omega_k$ and 
\eqn{
\label{eq:f1intcond}
\sup_{0<\abs{\gw-\gw_k}\leq \eps_A}\abs{\gw-\gw_k}^\ga f_k^1(i\gw)<\infty, \quad \mbox{and}, \quad
\sup_{\xi>0}\; \xi \int_{-\infty}^\infty f_k^1(\xi+i\eta)^2 d\eta<\infty, 
}
then the claim of the lemma is clearly satisfied with $f_k(\gl) = f_k^1(\gl)+f_k^0(\gl)$.

  We need to consider several different situations corresponding to different values of the exponents $\ga$, $\gb$, and $\gg$. 
  We have from Lemmas~\ref{lem:Okbdd} and~\ref{lem:ARfracbnd} that there exist $M_0,M_1\geq 1$ such that  $\sup_{\gl\in\bigcup_k\Omega_k}\abs{\gl_k}^\ga \norm{R(\gl,A)}\leq M_0$ and $\sup_{\gl\in\bigcup_k\Omega_k}\abs{\gl_k}^{\floor{\ga}} \norm{(-A_k)^{\tilde{\ga}} R(\gl,A)}\leq M_1$.

  \textit{Case}~1: If $\tilde{\ga}=0$, then the different possibilities are:
  \begin{itemize}
    \item[\textup{1.1.}] If $\tilde{\gb}=\tilde{\gg}=0$, then
      \eq{
      \abs{\gl_k}^{m+n} \norm{R_\gl B_m}\norm{R_\gl^\ast \tilde{C}_n} 
      \leq \abs{\gl_k}^\ga \norm{R_\gl} \norm{B_\gb}\norm{R_\gl^\ast \tilde{C}_\gg} 
     \leq M_0 \gd\norm{R_\gl^\ast \tilde{C}_\gg} 
     =:f_k^1(\gl)
      }
      and $f_k^1(\cdot)$ satisfies~\eqref{eq:f1intcond} due to Lemma~\ref{lem:BCfinrankint} and Assumption~\ref{ass:Astandass} (since we in particular have $f_k^1(i\gw)\leq M_0\gd \norm{\tilde{C}_\gg}\norm{R(i\gw,A)}$).
    \item[\textup{1.2.}] If $\tilde{\gb}+\tilde{\gg}=1$, and $\tilde{\gb}=1$, $\tilde{\gg}=0$, then
      \eq{
      \MoveEqLeft \abs{\gl_k}^{m+n} \norm{R_\gl B_m}\norm{R_\gl^\ast \tilde{C}_n} 
      =\abs{\gl_k}^{\ga-1} \norm{(-A_k)R_\gl B_\gb}\norm{R_\gl^\ast \tilde{C}_\gg} \\
      &\leq\abs{\gl_k}^{\ga} \norm{R_\gl} \norm{B_\gb}\norm{R_\gl^\ast \tilde{C}_\gg}  +  \abs{\gl_k}^{\ga-1} \norm{ B_\gb}\norm{R_\gl^\ast \tilde{C}_\gg}\\
&\leq (M_0+1)\gd\norm{R_\gl^\ast \tilde{C}_\gg}
     =:f_k^1(\gl) 
      }
      and $f_k^1(\cdot)$ satisfies~\eqref{eq:f1intcond} due to Lemma~\ref{lem:BCfinrankint} and Assumption~\ref{ass:Astandass}.
    \item[\textup{1.3.}] If $\tilde{\gb}+\tilde{\gg}=1$, and $\tilde{\gb}=0$, $\tilde{\gg}=1$, then the situation can be handled analogously to the case 1.2. We get 
      \eq{
      \MoveEqLeft \abs{\gl_k}^{m+n} \norm{R_\gl B_m}\norm{R_\gl^\ast \tilde{C}_n} 
\leq (M_0+1)\gd\norm{R_\gl B_\gb}
     =:f_k^1(\gl) 
      } 
      and $f_k^1(\cdot)$ satisfies~\eqref{eq:f1intcond} due to Lemma~\ref{lem:BCfinrankint} and Assumption~\ref{ass:Astandass}.
    \item[\textup{1.4.}] If $\tilde{\gb}+\tilde{\gg}=1$, and $0<\tilde{\gb},\tilde{\gg}<1$, then the moment inequality in Lemma~\ref{lem:Momentineq} implies
      \eq{
      \MoveEqLeft \abs{\gl_k}^{m+n} \norm{R_\gl B_m}\norm{R_\gl^\ast \tilde{C}_n} 
      = \abs{\gl_k}^{m+n} \norm{(-A_k)^{\tilde{\gb}} R_\gl B_\gb}\norm{(-A_k^\ast)^{\tilde{\gg}} R_\gl^\ast \tilde{C}_\gg} \\
      &\leq \abs{\gl_k}^{m+n} M_{\tilde{\gb}} \norm{ R_\gl B_\gb}^{1-\tilde{\gb}}\norm{(-A_k) R_\gl B_\gb}^{\tilde{\gb}} M_{\tilde{\gg}} \norm{ R_\gl^\ast \tilde{C}_\gg}^{1-\tilde{\gg}}\norm{(-A_k^\ast) R_\gl^\ast \tilde{C}_\gg}^{\tilde{\gg}} \\
      &\leq \abs{\gl_k}^{m+n} \norm{(-A_k) R_\gl}^{\tilde{\gb}+\tilde{\gg}} \norm{ B_\gb}^{\tilde{\gb}}\norm{ \tilde{C}_\gg}^{\tilde{\gg}} M_{\tilde{\gb}} M_{\tilde{\gg}} \norm{ R_\gl B_\gb}^{1-\tilde{\gb}}  \norm{ R_\gl^\ast \tilde{C}_\gg}^{1-\tilde{\gg}}\\
      &\leq(\abs{\gl_k}^{m+n+1} \norm{R_\gl}+1) \gd^{\tilde{\gb}+\tilde{\gg}} M_{\tilde{\gb}} M_{\tilde{\gg}} \norm{ R_\gl B_\gb}^{1-\tilde{\gb}}  \norm{ R_\gl^\ast \tilde{C}_\gg}^{1-\tilde{\gg}}\\
      &\leq(M_0+1) \gd M_{\tilde{\gb}} M_{\tilde{\gg}}\norm{ R_\gl B_\gb}^{1-\tilde{\gb}}  \norm{ R_\gl^\ast \tilde{C}_\gg}^{1-\tilde{\gg}}
     =:f_k^1(\gl) ,
      } 
      since $\tilde{\gb}+\tilde{\gg}=1$ and $\abs{\gl_k}\leq 1$.
      If $\gl=i\gw$ with $0<\abs{\gw-\gw_k}\leq \eps_A$, then 
      \eq{
\MoveEqLeft \abs{\gw-\gw_k}^\ga\norm{ R_\gl B_\gb}^{1-\tilde{\gb}}  \norm{ R_\gl^\ast \tilde{C}_\gg}^{1-\tilde{\gg}}
\leq \norm{B_\gb}^{1-\tilde{\gb}} \norm{ \tilde{C}_\gg}^{1-\tilde{\gg}} \abs{\gw-\gw_k}^\ga \norm{ R(i\gw,A)}^{1-\tilde{\gb}+1-\tilde{\gg}}  \\
&\leq \norm{B_\gb}^{1-\tilde{\gb}} \norm{ \tilde{C}_\gg}^{1-\tilde{\gg}} M_A
      }
      and thus $f_k^1(\cdot)$ satisfies the first part of~\eqref{eq:f1intcond}.
      Moreover, if we denote $q=1/(1-\tilde{\gb})$, $r = 1/(1-\tilde{\gg})$, then $1/q+1/r=1$ and the Hölder inequality implies
      \eq{
      \MoveEqLeft\sup_{\xi>0}\; \xi\int_{-\infty}^\infty \norm{ R(\xi+i\eta,A) B_\gb}^{2(1-\tilde{\gb})} \norm{ R(\xi+i\eta,A)^\ast \tilde{C}_\gg}^{2(1-\tilde{\gg})} d\eta\\
      &\leq \left(\sup_{\xi>0}\; \xi\int_{-\infty}^\infty \norm{ R(\xi+i\eta,A) B_\gb}^2 d\eta \right)^\frac{1}{q}
      \left(\sup_{\xi>0}\; \xi\int_{-\infty}^\infty \norm{ R(\xi+i\eta,A)^\ast \tilde{C}_\gg}^2 d\eta \right)^\frac{1}{r}\\
      &<\infty
      }
      by Lemma~\ref{lem:BCfinrankint}. This concludes that $f_k^1(\cdot)$ satisfies~\eqref{eq:f1intcond}.
  \end{itemize}

  \textit{Case}~2: If $\tilde{\ga}>0$ and $\tilde{\gb}+\tilde{\gg}=\tilde{\ga}<1$, then the different possibilities are: 
  \begin{itemize}
    \item[\textup{2.1.}] 
 If $\tilde{\gb}=\tilde{\ga}$ and $\tilde{\gg}=0$, then $\floor{\ga}=m+n$, and
      \eq{
      \MoveEqLeft \abs{\gl_k}^{m+n} \norm{R_\gl B_m}\norm{R_\gl^\ast \tilde{C}_n} 
      \leq \abs{\gl_k}^{\floor{\ga}} \norm{(-A_k)^{\tilde{\ga}} R_\gl}\norm{ B_\gb}\norm{ R_\gl^\ast \tilde{C}_\gg} \\
      &\leq M_1\gd\norm{ R_\gl^\ast \tilde{C}_\gg} 
     =:f_k^1(\gl) 
      }
      and $f_k^1(\cdot)$ satisfies~\eqref{eq:f1intcond} due to Lemma~\ref{lem:BCfinrankint} and Assumption~\ref{ass:Astandass}.
    \item[\textup{2.2.}] 
    If  $\tilde{\gb}=0$ and $\tilde{\gg}=\tilde{\ga}$, can be handled analogously to the case 2.1.  
\item[\textup{2.3.}] 
 If $0<\tilde{\gb},\tilde{\gg}<\tilde{\ga}$,
  then $\tilde{\gb}/\tilde{\ga}+\tilde{\gg}/\tilde{\ga}=1$. 
  The Moment Inequality in Lemma~\ref{lem:Momentineq} implies that there exists constans $M_{\tilde{\gb}/\tilde{\ga}},M_{\tilde{\gg}/\tilde{\ga}}$ not depending on $k\in\Igw$ such that
    \eq{
    \MoveEqLeft\abs{\gl_k}^{m+n} \norm{R_\gl B_m} \norm{R_\gl^\ast \tilde{C}_n}
    =\abs{\gl_k}^{\floor{\ga}} \norm{(-A_k)^{\tilde{\gb}} R_\gl B_\gb} \norm{(-A_k^\ast)^{\tilde{\gg}} R_\gl^\ast \tilde{C}_\gg}\\
    &\leq\abs{\gl_k}^{\floor{\ga}} M_{\tilde{\gb}/\tilde{\ga}}  \norm{R_\gl B_\gb}^{1-\tilde{\gb}/\tilde{\ga}}\norm{(-A_k)^{\tilde{\ga}} R_\gl B_\gb}^{\tilde{\gb}/\tilde{\ga}}  \\
    &\phantom{\leq}\; \times M_{\tilde{\gg}/\tilde{\ga}}  \norm{R_\gl^\ast \tilde{C}_\gg}^{1-\tilde{\gg}/\tilde{\ga}} \norm{(-A_k^\ast)^{\tilde{\ga}} R_\gl^\ast \tilde{C}_\gg}^{\tilde{\gg}/\tilde{\ga}}\\
    &\leq\abs{\gl_k}^{\floor{\ga}} \norm{(-A_k)^{\tilde{\ga}} R_\gl} M_{\tilde{\gb}/\tilde{\ga}} M_{\tilde{\gg}/\tilde{\ga}} \norm{B_\gb}^{\tilde{\gb}/\tilde{\ga}} \norm{\tilde{C}_\gg}^{\tilde{\gg}/\tilde{\ga}}    \norm{R_\gl B_\gb}^{1-\tilde{\gb}/\tilde{\ga}}  \norm{R_\gl^\ast \tilde{C}_\gg}^{1-\tilde{\gg}/\tilde{\ga}} \\
    &\leq M_1 M_{\tilde{\gb}/\tilde{\ga}} M_{\tilde{\gg}/\tilde{\ga}} \gd  \norm{R_\gl B_\gb}^{1-\tilde{\gb}/\tilde{\ga}}  \norm{R_\gl^\ast \tilde{C}_\gg}^{1-\tilde{\gg}/\tilde{\ga}} 
    =:f_k^1(\gl).
    }
    If we denote $q=1/(1-\tilde{\gb}/\tilde{\ga})$ and $r=1/(1-\tilde{\gg}/\tilde{\ga})$, then $1/q+1/r=1$ and we can show that $f_k^1(\cdot)$ satisfies~\eqref{eq:f1intcond} exactly as in the case 1.4.
\end{itemize}

\textit{Case}~3:
If $\tilde{\ga}>0$ and $\tilde{\gb}+\tilde{\gg}=\tilde{\ga}+1>1$, then necessarily $0<\tilde{\gb},\tilde{\gg}<1$, and we can choose $\gb_1=\tilde{\gb}/(\tilde{\ga}+1)$ and $\gg_1=\tilde{\gg}/(\tilde{\ga}+1)$. Then $0<\gb_1< \tilde{\gb}$ and $0<\gg_1< \tilde{\gg}$, and $\gb_1+\gg_1=1$. Using the Moment Inequality in Lemma~\ref{lem:Momentineq} we get
    \eq{
    \MoveEqLeft\abs{\gl_k}^{m+n} \norm{R_\gl B_m} \norm{R_\gl^\ast \tilde{C}_n}
    =\abs{\gl_k}^{m+n} \norm{(-A_k)^{\gb_1} R_\gl B_{m+\gb_1}} \norm{(-A_k^\ast)^{\gg_1} R_\gl^\ast \tilde{C}_{n+\gg_1}}\\
    &\leq\abs{\gl_k}^{m+n} M_{\gb_1}  \norm{R_\gl B_{m+\gb_1}}^{1-\gb_1} \norm{(-A_k) R_\gl B_{m+\gb_1}}^{\gb_1}  \\
    &\phantom{\leq}\; \times M_{\gg_1}  \norm{R_\gl^\ast \tilde{C}_{n+\gg_1}}^{1-\gg_1} \norm{(-A_k^\ast) R_\gl^\ast \tilde{C}_{n+\gg_1}}^{\gg_1}\\
    &= M_{\gb_1}  \abs{\gl_k}^{\gb_1(m+n)}\norm{R_\gl B_{m+\gb_1}}^{1-\gb_1} \norm{(-A_k) R_\gl B_{m+\gb_1}}^{\gb_1}  \\
    &\phantom{\leq}\; \times M_{\gg_1}  \abs{\gl_k}^{\gg_1(m+n)}\norm{R_\gl^\ast \tilde{C}_{n+\gg_1}}^{1-\gg_1} \norm{(-A_k^\ast) R_\gl^\ast \tilde{C}_{n+\gg_1}}^{\gg_1}.
    }
    Now
    \eq{
    \MoveEqLeft \abs{\gl_k}^{\gb_1(m+n)} \norm{R_\gl B_{m+\gb_1}}^{1-\gb_1} \norm{(-A_k)R_\gl B_{m+\gb_1}}^{\gb_1}\\
    &\leq \abs{\gl_k}^{\gb_1(m+n)}\norm{R_\gl B_{m+\gb_1}}^{1-\gb_1}\left(   \abs{\gl_k} \norm{R_\gl B_{m+\gb_1}} + \norm{B_{m+\gb_1}} \right)^{\gb_1} \\
    &\leq 2^{\gb_1} \left(   \abs{\gl_k}^{\gb_1 \floor{\ga}} \norm{R_\gl B_{m+\gb_1}} + \norm{R_\gl B_{m+\gb_1}}^{1-\gb_1}\norm{B_{m+\gb_1}}^{\gb_1} \right) 
    }
    since $\floor{\ga}=m+n+1$ and $\abs{\gl_k}\leq 1$.
     Using the Moment Inequality and the fact that $(\tilde{\gb}-\gb_1)/\tilde{\ga} = \tilde{\gb}(1-1/(\tilde{\ga}+1))/\tilde{\ga} = \gb_1$ we get
    \eq{
    \MoveEqLeft \abs{\gl_k}^{\gb_1 \floor{\ga}} \norm{R_\gl B_{m+\gb_1}} 
    =\abs{\gl_k}^{\gb_1 \floor{\ga}} \norm{(-A_k)^{\tilde{\gb}-\gb_1}R_\gl B_\gb} \\
    &\leq  \abs{\gl_k}^{\gb_1 \floor{\ga} } M_{\gb_1}\norm{R_\gl B_\gb}^{1-\gb_1} \norm{(-A_k)^{\tilde{\ga}}R_\gl B_\gb}^{\gb_1} \\
    &\leq   M_{\gb_1} \norm{B_\gb}^{\gb_1} \norm{R_\gl B_\gb}^{1-\gb_1} \left( \abs{\gl_k}^{ \floor{\ga} }\norm{(-A_k)^{\tilde{\ga}}R_\gl } \right)^{\gb_1} 
    \leq   M_{\gb_1} M_1^{\gb_1} \gd^{\gb_1} \norm{R_\gl B_\gb}^{1-\gb_1} .
    } 
    Analogously for the terms with $\tilde{C}_{n+\gg_1}$ we have 
    \eq{
    \MoveEqLeft \abs{\gl_k}^{\gg_1(m+n)} \norm{R_\gl^\ast \tilde{C}_{n+\gg_1}}^{1-\gg_1} \norm{(-A_k^\ast)R_\gl^\ast \tilde{C}_{n+\gg_1}}^{\gg_1}\\
    &\leq 2^{\gg_1} \left(   \abs{\gl_k}^{\gg_1 \floor{\ga}} \norm{R_\gl^\ast \tilde{C}_{n+\gg_1}} + \abs{\gl_k}^{\gg_1(m+n)}\norm{R_\gl^\ast \tilde{C}_{n+\gg_1}}^{1-\gg_1}\norm{\tilde{C}_{n+\gg_1}}^{\gg_1} \right) \\
    &\leq 2^{\gg_1} \left( M_{\gg_1} M_1^{\gg_1} \gd^{\gg_1} \norm{R_\gl^\ast \tilde{C}_\gg}^{1-\gg_1} + \norm{R_\gl^\ast \tilde{C}_{n+\gg_1}}^{1-\gg_1}\norm{\tilde{C}_{n+\gg_1}}^{\gg_1} \right) .
    }

Combining these estimates we can see that if we choose
\eq{
\tilde{M} = 2M_{\gb_1}M_{\gg_1}  \max \set{M_{\gb_1} M_1^{\gb_1} \gd^{\gb_1},\norm{B_{m+\gb_1}}^{\gb_1}} \cdot \max\set{M_{\gg_1} M_1^{\gg_1} \gd^{\gg_1},\norm{\tilde{C}_{n+\gg_1}}^{\gg_1}} ,
}
 then (using $\gb_1+\gg_1=1$)
\eq{
    \MoveEqLeft\abs{\gl_k}^{m+n} \norm{R_\gl B_m} \norm{R_\gl^\ast \tilde{C}_n}\\
    &\leq 2^{\gb_1+\gg_1}M_{\gb_1}M_{\gg_1} \left(M_{\gb_1} M_1^{\gb_1} \gd^{\gb_1} \norm{R_\gl B_\gb}^{1-\gb_1}    + \norm{R_\gl B_{m+\gb_1}}^{1-\gb_1}\norm{B_{m+\gb_1}}^{\gb_1} \right) \\
&\phantom{\leq} ~\times\left( M_{\gg_1} M_1^{\gg_1} \gd^{\gg_1} \norm{R_\gl^\ast \tilde{C}_\gg}^{1-\gg_1} + \norm{R_\gl^\ast \tilde{C}_{n+\gg_1}}^{1-\gg_1}\norm{\tilde{C}_{n+\gg_1}}^{\gg_1} \right) \\
&\leq \tilde{M} \left(\norm{R_\gl B_\gb}^{1-\gb_1}    + \norm{R_\gl B_{m+\gb_1}}^{1-\gb_1} \right) 
\left( \norm{R_\gl^\ast \tilde{C}_\gg}^{1-\gg_1} + \norm{R_\gl^\ast \tilde{C}_{n+\gg_1}}^{1-\gg_1} \right) \\
&=  \tilde{M} \bigl(
\norm{R_\gl B_\gb}^{1-\gb_1} \norm{R_\gl^\ast \tilde{C}_\gg}^{1-\gg_1}   
+ \norm{R_\gl B_\gb}^{1-\gb_1} \norm{R_\gl^\ast \tilde{C}_{n+\gg_1}}^{1-\gg_1} \\
&\phantom{\leq} ~+ \norm{R_\gl B_{m+\gb_1}}^{1-\gb_1} \norm{R_\gl^\ast \tilde{C}_\gg}^{1-\gg_1} 
+ \norm{R_\gl B_{m+\gb_1}}^{1-\gb_1} \norm{R_\gl^\ast \tilde{C}_{n+\gg_1}}^{1-\gg_1} 
\bigr)
=:f_k^1(\gl).
}
Denoting $q=1/(1-\gb_1)$ and $r=1/(1-\gg_1)$, we have $1/q+1/r=1$, and we can use the scalar inequality $(\sum_{j=1}^4 a_j)^2\leq 4 \sum_{j=1}^4 a_j^2$ for $a_j\geq 0$ to show that $f_k^1(\gl)$ satisfies~\eqref{eq:f1intcond} similarly as in the case~1.4.
\end{proof}

\begin{proof}[Proof of Theorem~\textup{\ref{thm:stabpert}}]
  Let $\gd>0$ be chosen as in Theorem~\ref{thm:specpert} and assume $\norm{B}<\gd$, $\norm{C}<\gd$, and $\norm{(i\gw_k-A)^{-\gb}B}<\gd$, and $\norm{(-i\gw_k-A^\ast)^{-\gg}C^\ast}<\gd$ for all $k\in\Igw$. 
  By Theorem~\ref{thm:specpert} there exists $M_D\geq 1$ such that $\norm{(I-CR(\gl,A)B)\inv}\leq M_D$ for all $\gl\in \overline{\C^+}\setminus \set{i\gw_k}_{k\in\Igw}$. 
  We begin the proof by showing that the semigroup generated by $A+BC$ is uniformly bounded.

Let $x\in X$ and
for brevity denote $R_\gl = R(\xi+i\eta,A)$ and $D_\gl = I-CR(\xi+i\eta,A)B$. 
Using Shermann--Morrison--Woodbury formula in Lemma~\ref{lem:ShermanMorrisonWoodbury} and the scalar inequality $(a+b)^2\leq 2 (a^2+b^2)$ for $a,b\geq 0$
we get
\eq{
\MoveEqLeft \sup_{\xi>0}\; \xi \int_{-\infty}^\infty \norm{R(\xi+i\eta,A+BC)x}^2 d\eta
= \sup_{\xi>0}\;\xi \int_{-\infty}^\infty \norm{R_\gl x + R_\gl BD_\gl\inv CR_\gl x}^2 d\eta\\
&\leq 2\sup_{\xi>0}\;\xi \int_{-\infty}^\infty \norm{R_\gl x}^2 + \norm{R_\gl B}^2 \norm{D_\gl\inv}^2 \norm{ CR_\gl}^2 \norm{ x}^2 d\eta\\
&\leq 2\sup_{\xi>0}\;\xi \int_{-\infty}^\infty \norm{R_\gl x}^2 d\eta + 2M_D^2 \norm{x}^2 \sup_{\xi>0}\;\xi \int_{-\infty}^\infty \norm{R_\gl B}^2  \norm{ CR_\gl}^2 d\eta .
}
Similarly, using $\norm{(R_\gl BD_\gl\inv CR_\gl)^\ast} = \norm{R_\gl BD_\gl\inv CR_\gl}\leq M_D\norm{R_\gl B}\norm{CR_\gl}$ we get
\eq{
\MoveEqLeft \sup_{\xi>0}\; \xi \int_{-\infty}^\infty \norm{R(\xi+i\eta,A+BC)^\ast x}^2 d\eta
= \sup_{\xi>0}\;\xi \int_{-\infty}^\infty \norm{R_\gl^\ast x + (R_\gl BD_\gl\inv CR_\gl)^\ast x}^2 d\eta\\
&\leq 2\sup_{\xi>0}\;\xi \int_{-\infty}^\infty \norm{R_\gl^\ast x}^2 d\eta + 2M_D^2 \norm{x}^2 \sup_{\xi>0}\;\xi \int_{-\infty}^\infty \norm{R_\gl B}^2  \norm{ CR_\gl}^2 d\eta 
}
In both cases the first supremums are finite by Theorem~\ref{thm:unifbddconds}.
Therefore, Theorem~\ref{thm:unifbddconds} implies that the semigroup generated by $A+BC$ is uniformly bounded if 
\eqn{
\label{eq:stabpertRBCRint}
\sup_{\xi>0}\;\xi \int_{-\infty}^\infty \norm{R_\gl B}^2  \norm{ CR_\gl}^2 d\eta 
<\infty.
}

For all $k\in\Igw$ let $f_k(\cdot)$ be the functions in Lemma~\ref{lem:RBCRest}. By Lemma~\ref{lem:RbddOcomp} we can choose $M_2\geq 1$ such that $\norm{R(\gl,A)}\leq M_2$ for all $\gl\in \overline{\C^+}\setminus \left( \bigcup_k \Omega_k \right)$.

Let $\xi>0$. For each $k\in\Igw$ denote by $E_k^\xi\subset \R$ the interval such that $\xi+i\eta \in\Omega_k$ if and only if $\eta\in E_k^\xi$. Finally, denote $E^\xi = \R\setminus \left( \bigcup_k E_k^\xi \right)$. Now, since the set $\Igw$ of indices is finite, we have
\eq{
\MoveEqLeft\xi \int_{-\infty}^\infty \norm{R_\gl B}^2  \norm{ CR_\gl}^2 d\eta 
= \xi \int_{E^\xi} \norm{R_\gl B}^2  \norm{ CR_\gl}^2 d\eta 
+ \sum_{k\in\Igw} \xi \int_{E_k^\xi} \norm{R_\gl B}^2  \norm{ CR_\gl}^2 d\eta \\
&\leq \xi \int_{E^\xi} M_2^2 \norm{B}^2  \norm{ CR_\gl}^2 d\eta 
+ \sum_{k\in\Igw} \xi \int_{E_k^\xi} f_k(\xi+i\eta)^2 d\eta \\
&\leq M_2^2 \norm{B}^2 \xi \int_{-\infty}^\infty   \norm{ CR_\gl}^2 d\eta 
+ \sum_{k\in\Igw} \xi \int_{-\infty}^\infty f_k(\xi+i\eta)^2 d\eta \\
&\leq M_2^2 \norm{B}^2 \sup_{\xi>0}\;\xi \int_{-\infty}^\infty   \norm{ CR_\gl}^2 d\eta 
+ \sum_{k\in\Igw} \sup_{\xi>0}\;\xi \int_{-\infty}^\infty f_k(\xi+i\eta)^2 d\eta <\infty
}
by Lemmas~\ref{lem:BCfinrankint}~and~\ref{lem:RBCRest}. Since the bound is independent of $\xi>0$, this shows that~\eqref{eq:stabpertRBCRint} is satisfied, and thus concludes that the semigroup generated by $A+BC$ is uniformly bounded.

Since the perturbed semigroup is uniformly bounded and $X$ is a Hilbert space, the Mean Ergodic Theorem~\cite{arendtbatty} implies that $\gs(A+BC)\cap i\R\subset \gs_p(A+BC)\cup \gs_c(A+BC)$. However, by Theorem~\ref{thm:specpert} we have that
$i\gw_k\notin\gs_p(A+BC)$ for all $k\in \Igw$. This concludes that $i\gw_k\in \gs_c(A+BC)\cup \rho(A+BC)$ for all $k\in\Igw$.

Theorem~\ref{thm:specpert} shows that $\gs(A+BC)\cap i\R \subset \set{i\gw_k}_{k\in\Igw}$ is countable and $\gs_p(A+BC)\cap i\R = \varnothing$. The Arent--Batty--Lyubich--V\~{u} Theorem~\cite{arendtbattystrongstab,lyubichvu} therefore concludes that the semigroup generated by $A+BC$ is strongly stable.

It remains to show that for all $k\in\Igw$ the resolvent operator $R(\gl,A+BC)$ satisfies
\eqn{
\label{eq:stabpertgrowthest}
\sup_{0<\abs{\gw-\gw_k}\leq \eps_A} \abs{\gw-\gw_k}^\ga \norm{R(i\gw,A+BC)}<\infty.
} 
To this end, let $k\in \Igw$ be arbitary. By Lemma~\ref{lem:RBCRest} there exists $M_k\geq 1$ such that $\abs{\gw-\gw_k}^\ga f_k(i\gw)\leq M_k$ whenever $0<\abs{\gw-\gw_k}\leq\eps_A$. 
The Shermann--Morrison--Woodbury formula in Lemma~\ref{lem:ShermanMorrisonWoodbury} implies that for all $\gw\in\R$ satisfying $0<\abs{\gw-\gw_k}\leq \eps_A$ we have
\eq{
\MoveEqLeft\norm{R(i\gw,A+BC)}
=\norm{R(i\gw,A)  + R(i\gw,A) B(I-CR(i\gw,A)B)\inv CR(i\gw,A)}\\
&\leq\norm{R(i\gw,A)}  + \norm{R(i\gw,A) B} \norm{(I-CR(i\gw,A)B)\inv} \norm{CR(i\gw,A)} \\
&\leq\norm{R(i\gw,A)}  + M_D f_k(i\gw ),
}
and thus
\eq{
\abs{\gw-\gw_k}^\ga \norm{R(i\gw,A+BC)}
&\leq\abs{\gw-\gw_k}^\ga\norm{R(i\gw,A)}  + M_D \abs{\gw-\gw_k}^\ga  f_k(i\gw )
\leq M_A + M_D M_k.
}
This concludes that~\eqref{eq:stabpertgrowthest} is satisfied. 
On the other hand, if $\abs{\gw-\gw_k}>\eps_A$ for all $k\in\Igw$, then
\eq{
\MoveEqLeft\norm{R(i\gw,A+BC)}
\leq\norm{R(i\gw,A)}  + \norm{R(i\gw,A) B} \norm{(I-CR(i\gw,A)B)\inv} \norm{CR(i\gw,A)} \\
&\leq\norm{R(i\gw,A)}  + M_D \norm{B}\norm{C}\norm{R(i\gw,A)}^2
\leq M_A+M_D \norm{B} \norm{C} M_A^2,
}
and thus $\norm{R(i\gw,A+BC)}$ is uniformly bounded for $\gw\in\R$ satisfying $\abs{\gw-\gw_k}>\eps_A$ for all $k\in\Igw$.
This concludes the proof.  
\end{proof}

\section{Preservation of Polynomial Stability}
\label{sec:polpert}

In this section we prove Theorem~\ref{thm:Apolpert}, which gives conditions for the preservation of polynomial stability of a semigroup under finite rank perturbations.

\begin{proof}[{Proof of Theorem~\textup{\ref{thm:Apolpert}}}]
  If $\gb=0$ or $\gg=0$, the claim follows directly from~\cite[Thm. 5]{Pau12}. We can therefore assume $\gb,\gg>0$.

  Choose $\gd = \sqrt{\gd_1}>0$, where $\gd_1>0$ is chosen as in \cite[Cor. 7]{Pau13a}. Assume $B\in \Lin(\C^p,X)$ and $C\in \Lin(X,\C^p)$ satisfy~\eqref{eq:BCpolstabrancond} and $\norm{(-A)^\gb B}<\gd$  and $\norm{(-A^\ast)^{\gg} C^\ast}<\gd$. Since $\norm{(-A)^\gb B}\cdot\norm{(-A^\ast)^{\gg} C^\ast}<\gd_1$, we have from \cite[Cor. 7]{Pau13a} that $\gs(A+BC)\subset \C^-$, $1\in \rho(CR(\gl,A)B)$ for all $\gl\in \overline{\C^+}$,  and if we denote $D_\gl = I-CR(\gl,A)B$, then there exists $M_D\geq 1$ such that
  \eq{
  \sup_{\gl\in \overline{\C^+}} \norm{D_\gl\inv}\leq M_D<\infty.
  }
  Since $A$ generates a polynomially stable semigroup, Theorem 2.4 and Lemma 2.3 in~\cite{BorTom10} show that we can choose $M_R\geq 1$ in such a way that $\norm{R(\gl,A)(-A)^{-\ga}}\leq M_R$ for all $\gl\in \overline{\C^+}$.
  
  We begin by showing that 
  \eqn{
  \label{eq:polRBCRint}
  \sup_{\xi>0} \; \xi\int_{-\infty}^\infty \norm{R(\xi+i\eta,A)B}^2 \norm{CR(\xi+i\eta,A)}^2 d\eta<\infty
  }
  and 
  \eqn{
  \label{eq:polRBCRpolbdd}
  \norm{R(i\gw,A)B} \norm{CR(i\gw,A)}=\Omi(\abs{\gw}^\ga).
  } 
  To this end, choose $0<\gb_1\leq \gb$ and $0<\gg_1\leq \gg$ such that $\gb_1+\gg_1=\ga$.
  Let $\gl\in \overline{\C^+}$ and denote $R_\gl = R(\gl,A)$, $B_{\gb_1} = (-A)^{\gb_1} B$, and $\tilde{C}_{\gg_1} =(-A^\ast)^{\gg_1} C^\ast$. 
  The Moment Inequality~\cite[Prop. 6.6.4]{haasefuncalc} implies that there exist constants $M_{\gb_1/\ga}$ and $M_{\gg_1/\ga}$ such that
  \eq{
  \norm{R_\gl B}
  &=\norm{R_\gl (-A)^{-\gb_1}(-A)^{\gb_1} B}
  =\norm{(-A)^{-\gb_1}R_\gl B_{\gb_1}}\\
  &\leq M_{\gb_1/\ga} \norm{R_\gl B_{\gb_1}}^{1-\gb_1/\ga}\norm{(-A)^{-\ga}R_\gl B_{\gb_1}}^{\gb_1/\ga}\\
  &\leq M_{\gb_1/\ga} \norm{R_\gl B_{\gb_1}}^{1-\gb_1/\ga}\norm{(-A)^{-\ga}R(\gl,A)}^{\gb_1/\ga} \norm{ B_{\gb_1}}^{\gb_1/\ga}\\
  &\leq M_{\gb_1/\ga} \norm{R_\gl B_{\gb_1}}^{1-\gb_1/\ga}M_R^{\gb_1/\ga}  \norm{ B_{\gb_1}}^{\gb_1/\ga}
  }
  and
  \eq{
  \norm{CR_\gl }
  &=\norm{R_\gl^\ast C^\ast}
  =\norm{R_\gl^\ast (-A^\ast)^{-\gg_1}(-A^\ast)^{\gg_1} C^\ast}
  =\norm{(-A^\ast)^{-\gg_1}R_\gl^\ast \tilde{C}_{\gg_1}}\\
  &\leq M_{\gg_1/\ga} \norm{R_\gl^\ast \tilde{C}_{\gg_1}}^{1-\gg_1/\ga}\norm{(-A^\ast)^{-\ga}R_\gl^\ast \tilde{C}_{\gg_1}}^{\gg_1/\ga}\\
  &\leq M_{\gg_1/\ga} \norm{R_\gl^\ast \tilde{C}_{\gg_1}}^{1-\gg_1/\ga}\norm{(-A)^{-\ga}R(\gl,A)}^{\gg_1/\ga} \norm{\tilde{C}_{\gg_1}}^{\gg_1/\ga}\\
  &\leq M_{\gg_1/\ga} \norm{R_\gl^\ast \tilde{C}_{\gg_1}}^{1-\gg_1/\ga} M_R^{\gg_1/\ga} \norm{\tilde{C}_{\gg_1}}^{\gg_1/\ga}.
  }
  If we choose $\tilde{M} = M_{\gb_1/\ga} M_{\gg_1/\ga} M_R^{(\gb_1+\gg_1)/\ga}  \norm{ B_{\gb_1}}^{\gb_1/\ga}\norm{ \tilde{C}_{\gg_1}}^{\gg_1/\ga}$, then
  \eq{
  \norm{R_\gl B} \norm{CR_\gl}
  \leq \tilde{M} \norm{R_\gl B_{\gb_1}}^{1-\gb_1/\ga} \norm{R_\gl^\ast \tilde{C}_{\gg_1}}^{1-\gg_1/\ga}.
  }
 Choose $q=1/(1-\gb_1/\ga)$ and $r=1/(1-\gg_1/\ga)$. Then
 $1/q+1/r
= 2-(\gb_1+\gg_1)/\ga = 1$, and 
 using the Hölder inequality we get
      \eq{
      \MoveEqLeft\sup_{\xi>0}\; \xi\int_{-\infty}^\infty \norm{ R(\xi+i\eta,A) B}^2 \norm{C R(\xi+i\eta,A)}^2 d\eta\\
      &\leq \tilde{M}^2\sup_{\xi>0}\; \xi\int_{-\infty}^\infty \norm{ R(\xi+i\eta,A) B_{\gb_1}}^{2(1-\gb_1/\ga)} \norm{ R(\xi+i\eta,A)^\ast \tilde{C}_{\gg_1}}^{2(1-\gg_1/\ga)} d\eta\\
      &\leq \tilde{M}^2 \left(\sup_{\xi>0}\; \xi\int_{-\infty}^\infty \norm{ R(\xi+i\eta,A) B_{\gb_1}}^2 d\eta \right)^q  \\
      &\qquad\quad \times\left(\sup_{\xi>0}\; \xi\int_{-\infty}^\infty \norm{ R(\xi+i\eta,A)^\ast \tilde{C}_{\gg_1}}^2 d\eta \right)^r
      <\infty
      }
      by Lemma~\ref{lem:BCfinrankint}. This concludes~\eqref{eq:polRBCRint}. Moreover, for $\gw\in\R$ with large $\abs{\gw}$ we have
      \eq{
  \MoveEqLeft\norm{R(i\gw,A) B} \norm{CR(i\gw,A)}
  \leq \tilde{M} \norm{R(i\gw,A) B_{\gb_1}}^{1-\gb_1/\ga} \norm{R(i\gw,A)^\ast \tilde{C}_{\gg_1}}^{1-\gg_1/\ga}  \\
  &\leq \tilde{M} \norm{R(i\gw,A)}^{1-\gb_1/\ga} \norm{ B_{\gb_1}}^{1-\gb_1/\ga} \norm{R(i\gw,A)^\ast}^{1-\gg_1/\ga} \norm{\tilde{C}_{\gg_1}}^{1-\gg_1/\ga}  \\
  &= \tilde{M} \norm{ B_{\gb_1}}^{1-\gb_1/\ga}  \norm{\tilde{C}_{\gg_1}}^{1-\gg_1/\ga}\norm{R(i\gw,A)}^{1-\gb_1/\ga+1-\gg_1/\ga}  
  = \Omi(\abs{\gw}^\ga),
      }
      since $1-\gb_1/\ga +1-\gg_1/\ga = 2-(\gb_1+\gg_1)/\ga = 1$ and $\norm{R(i\gw,A)}=\Omi(\abs{\gw}^\ga)$ by~\cite[Thm. 2.4]{BorTom10}. This concludes~\eqref{eq:polRBCRpolbdd}.

      We can now show that the semigroup generated by $A+BC$ is uniformly bounded.
  Let $x\in X$. 
  Using the Shermann--Morrison--Woodbury formula in Lemma~\ref{lem:ShermanMorrisonWoodbury} 
  we can estimate (exactly as in the proof of Theorem~\ref{thm:stabpert}) 
\eq{
\MoveEqLeft \sup_{\xi>0}\; \xi \int_{-\infty}^\infty \norm{R(\xi+i\eta,A+BC)x}^2 d\eta
= \sup_{\xi>0}\;\xi \int_{-\infty}^\infty \norm{R_\gl x + R_\gl BD_\gl\inv CR_\gl x}^2 d\eta\\
&\leq 2\sup_{\xi>0}\;\xi \int_{-\infty}^\infty \norm{R_\gl x}^2 + \norm{R_\gl B}^2 \norm{D_\gl\inv}^2 \norm{ CR_\gl}^2 \norm{ x}^2 d\eta\\
&\leq 2\sup_{\xi>0}\;\xi \int_{-\infty}^\infty \norm{R_\gl x}^2 d\eta + 2M_D^2 \norm{x}^2 \sup_{\xi>0}\;\xi \int_{-\infty}^\infty \norm{R_\gl B}^2  \norm{ CR_\gl}^2 d\eta 
<\infty
}
due to Theorem~\ref{thm:unifbddconds} and~\eqref{eq:polRBCRint}.
Analogously, we have
\eq{
\MoveEqLeft \sup_{\xi>0}\; \xi \int_{-\infty}^\infty \norm{R(\xi+i\eta,A+BC)^\ast x}^2 d\eta\\
&\leq 2\sup_{\xi>0}\;\xi \int_{-\infty}^\infty \norm{R_\gl^\ast x}^2 d\eta + 2M_D^2 \norm{x}^2 \sup_{\xi>0}\;\xi \int_{-\infty}^\infty \norm{R_\gl B}^2  \norm{ CR_\gl}^2 d\eta 
<\infty
}
again due to Theorem~\ref{thm:unifbddconds} and~\eqref{eq:polRBCRint}.
Since $x\in X$ was arbitrary, Theorem~\ref{thm:unifbddconds} concludes that the semigroup generated by $A+BC$ is uniformly bounded.

Finally, the Shermann-Morrison--Woodbury formula in Lemma~\ref{lem:ShermanMorrisonWoodbury} together with~\eqref{eq:polRBCRpolbdd} implies that for $\gw\in\R$ with large $\abs{\gw}$ we have
\eq{
\MoveEqLeft\norm{R(i\gw,A+BC)} = 
\norm{R(i\gw,A) + R(i\gw,A) B(I- CR(i\gw,A)B)\inv C  R(i\gw,A)}\\
&\leq \norm{R(i\gw,A)} + \norm{ R(i\gw,A) B}  \norm{(I- CR(i\gw,A)B)\inv} \norm{C R(i\gw,A) }\\
&\leq \norm{R(i\gw,A)} + M_D\norm{ R(i\gw,A) B} \norm{C R(i\gw,A) }
= \Omi(\abs{\gw}^\ga).
}
By Theorem~2.4 in~\cite{BorTom10} this concludes that the semigroup generated by $A+BC$ is polynomially stable with exponent $\ga$.
\end{proof}

\section{Perturbation of a Strongly Stable Multiplication Semigroup}
\label{sec:multex}

In this section we apply our theoretic results in considering the preservation of strong stability of a multiplication 
semigroup~\cite[Par. II.2.9]{engelnagel}
\eq{
(T_A(t) f)(\mu ) = e^{t\cdot \mu} f(\mu)
}
on $X=\Lp[2](\Omega)$, where $\Omega = \setm{\gl}{\abs{\gl+1}\leq 1}$ is a disk centered at $-1$ and with radius~$1$ (see Figure~\ref{fig:Multexdom}).
The generator $A$ of the semigroup $T_A(t)$ is a bounded multiplication operator  $(A f)(\mu) = \mu f(\mu)$. 

\begin{figure}[ht]
  \begin{center}
    \includegraphics[width=0.4\linewidth]{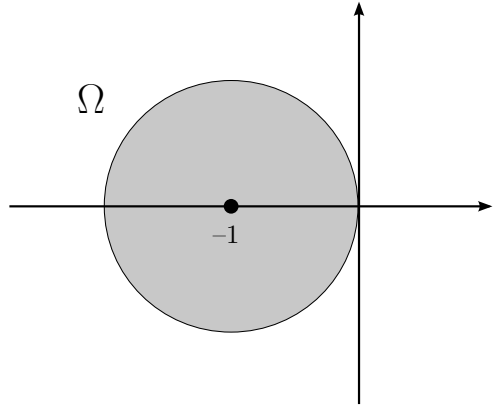}
  \end{center}
  \caption{The domain $\Omega$.}
  \label{fig:Multexdom}
\end{figure}

The spectrum of $A$ is given by $\gs(A)=\gs_c(A)=\Omega$, $\gs(A)\cap i\R = \set{0}\subset \gs_c(A)$, and the semigroup is uniformly bounded. Due to the Arendt--Batty--Lyubich--V\~{u} Theorem~\cite{arendtbattystrongstab} the semigroup $T_A(t)$ is strongly stable. The operator $-A$ has an unbounded inverse $(-A)\inv$ with domain
\eq{
\Dom( (-A)\inv) = \Bigl\{f\in \Lp[2](\Omega)\, \Bigm| \,\int_\Omega \abs{\mu}^{-2} \abs{f(\mu)}^2 d\mu < \infty\Bigr\}. 
} 

We begin by finding a suitable value for $\ga\geq 1$ in Assumption~\ref{ass:Astandass}. Due to the geometry, for all $\gw\in\R$ with $0<\abs{\gw}\leq 1=:\eps_A$ we have
\eq{
\MoveEqLeft\abs{\gw}^\ga \norm{R(i\gw,A)}
= \frac{\abs{\gw}^\ga}{\dist(i\gw,\Omega)} 
= \frac{\abs{\gw}^\ga}{\dist(i\gw,-1)-1} 
= \frac{\abs{\gw}^\ga}{\sqrt{\gw^2+1}-1} 
= \frac{\abs{\gw}^\ga(\sqrt{\gw^2+1}+1)}{\gw^2+1-1} \\
&= \frac{\abs{\gw}^\ga}{\gw^2}(\sqrt{\gw^2+1}+1)\leq M_A<\infty
}
if and only if $\ga\geq 2$.  Thus we can choose $\ga=2$ in Assumption~\ref{ass:Astandass}.

For $\gb \geq 0$ the domains of the operators $(-A)^{-\gb}$ and $(-A^\ast)^{-\gb}$ are given by 
\eq{
\Dom( (-A)^{-\gb}) = 
\Dom( (-A^\ast)^{-\gb}) = 
\Bigl\{f\in \Lp[2](\Omega)\, \Bigm| \,\int_{\Omega} \abs{\mu}^{-2\gb} \abs{f(\mu)}^2 d\mu <\infty\Bigr\}.
} 
If $0<\tilde{\gb}<\gb$ and $f\in \Dom( (-A)^{-\gb})$, then the Hölder inequality with exponents $q=1/(1-\tilde{\gb}/\tilde{\gb})$ and $r=1/(\tilde{\gb}/\gb)$ implies
\eq{
\MoveEqLeft \norm{(-A)^{-\tilde{\gb}}f}^2
=  \int_{\Omega} \abs{\mu}^{-2\tilde{\gb}} \abs{f(\mu)}^2 d\mu
=  \int_{\Omega} \abs{f(\mu)}^{2(1-\tilde{\gb}/\gb)} \cdot (\abs{\mu}^{-2\gb} \abs{f(\mu)}^2)^{\tilde{\gb}/\gb} d\mu\\
&\leq \left(  \int_{\Omega} \abs{f(\mu)}^{2} d\mu \right)^{1-\tilde{\gb}/\gb}  \left( \int_{\Omega}\abs{\mu}^{-2\gb} \abs{f(\mu)}^2 d\mu \right)^{\tilde{\gb}/\gb},
}
or equivalently $\norm{(-A)^{-\tilde{\gb}}f}\leq \norm{f}^{1-\tilde{\gb}/\gb}\norm{(-A)^{-\gb}f}^{\tilde{\gb}/\gb}$. In particular, this shows that if $\norm{B}<\gd$ and $\norm{(-A)^{-\gb}B}<\gd$, then for all $0<\tilde{\gb}<\gb$ we have
\eqn{
\label{eq:multexMoment}
\norm{(-A)^{-\tilde{\gb}}B}
\leq \norm{B}^{1-\tilde{\gb}/\gb}\norm{(-A)^{-\gb}B}^{\tilde{\gb}/\gb}
< \gd^{1-\tilde{\gb}/\gb} \gd^{\tilde{\gb}/\gb}
=\gd,
}
and similarly for $\norm{(-A^\ast)^{-\tilde{\gg}}C^\ast}$ with $0<\tilde{\gg}<\gg$.

We consider bounded finite rank perturbations $A+BC$.
Since the operator $A$ is bounded, we can approach the preservation of the strong stability of $T_A(t)$ more directly than in the proof of Theorem~\ref{thm:stabpert}. In particular, if $0<c<1$, then the theory presented in the earlier sections shows that the strong stability is preserved for all $B$ and $C$ for which
\eq{
\norm{CR(\gl,A)B}\leq c<1
}
for all $\gl\in \overline{\C^+}\setminus \set{0}$. In particular, if $\gb,\gg\geq 0$ are such that $\gb+\gg=2$ and Assumption~\ref{ass:Astandass} is satisfied, then 
\eq{
\MoveEqLeft\norm{CR(\gl,A)B}
= \sup_{\norm{x}=\norm{y}=1} \abs{\iprod{R(\gl,A)Bx}{C^\ast y}}\\
&= \sup_{\norm{x}=\norm{y}=1} \abs{\iprod{(-A)^2R(\gl,A)(-A)^{-\gb}Bx}{(-A^\ast)^{-\gg} C^\ast y}}\\
&\leq \norm{(-A)^2 R(\gl,A)}\norm{(-A)^{-\gb}B} \norm{(-A^\ast)^{-\gg}C^\ast}
\leq c<1
}
if $\norm{(-A)^{-\gb}B}<\sqrt{c/M_1}$ and $\norm{(-A^\ast)^{-\gg}C^\ast}<\sqrt{c/M_1}$ 
where $M_1\geq 1$ is such that
$\norm{(-A)^2 R(\gl,A)}\leq M_1$ for all $\gl\in \overline{\C^+}\setminus\set{0}$. In the following we will search for a suitable constant $M_1\geq 1$.

If $\gl\in \overline{\C^+}\setminus\set{0}$ is such that $\abs{\gl}\geq 2$, then
 for all $f\in X$ with $\norm{f}=1$ we have
\eq{
\MoveEqLeft\norm{(-A)^2 R(\gl,A)f}
=\left(  \int_{\Omega} \frac{\abs{\mu}^4 \abs{f(\mu)}^2}{\abs{\gl-\mu}^2}d\mu \right)^{1/2}
\leq \sup_{\mu\in\Omega} \frac{\abs{\mu}^2 }{\abs{\gl-\mu}}
\left( \int_{\Omega} \abs{f(\mu)}^2d\mu \right)^{1/2}\\
&= \sup_{\mu\in\Omega} \frac{\abs{\mu}^2 }{\abs{\gl-\mu}}
\leq \sup_{\mu\in\Omega} \frac{4 }{\abs{\gl-\mu}}
= \frac{4 }{\dist(\gl,\Omega)}
=\frac{4 }{\abs{\gl+1}-1}
\leq 4,
}
since $\abs{\gl+1}-1\geq \abs{\gl}-1\geq 2-1=1$. On the other hand, if $\gl\in \overline{\C^+}\setminus\set{0}$ and $\abs{\gl}< 2$, then
\eq{
\MoveEqLeft\norm{(-A)^2 R(\gl,A)}
=\norm{(-A)(I-\gl R(\gl,A))}
\leq\norm{A}+\abs{\gl}\norm{(-A) R(\gl,A)}\\
&\leq\norm{A}+\abs{\gl}+\abs{\gl}^2\norm{ R(\gl,A)}
\leq 2+2+\abs{\gl}^2 \dist(\gl,\Omega)^{-1}
}
where
\eq{
\MoveEqLeft\frac{\abs{\gl}^2 }{\dist(\gl,\Omega)}
=\frac{\abs{\gl}^2 }{\abs{\gl+1}-1}
=\frac{\abs{\gl}^2(\abs{\gl+1}+1) }{\abs{\gl+1}^2-1}
\leq\frac{\abs{\gl}^2(\abs{\gl}+2) }{(\re\gl+1)^2 +(\im\gl)^2-1}\\
&\leq\frac{4\abs{\gl}^2 }{(\re\gl)^2 +\re\gl +(\im\gl)^2}
\leq\frac{4\abs{\gl}^2 }{(\re\gl)^2  +(\im\gl)^2}
=4.
}
Together these estimates imply that if we choose $M_1 = 8$, then $\norm{(-A)^2R(\gl,A)}\leq M_1$ for all $\gl\in \overline{\C^+}\setminus \set{0}$.

In order to guarantee $0\notin \gs_p(A+BC)$, the choice for $\gd>0$ must satisfy the conditions of Lemma~\ref{lem:ABCinj}. 
The proof of the lemma shows that the appropriate condition for $B$ and $C$ is that $\norm{(-A)^{-\gb_1}B}\cdot \norm{(-A^\ast)^{-\gg_1}C^\ast}<1$ for some $0\leq \gb_1\leq \gb$ and $0\leq \gg_1\leq \gg$ satsifying $\gb_1+\gg_1=1$. Due to the property~\eqref{eq:multexMoment}, this is true whenever $\gd\leq 1$ and the perturbation satisfies $\norm{B}<\gd$, $\norm{C}<\gd$, $\norm{(-A)^{-\gb} B}<\gd$, and $\norm{(-A^\ast)^{-\gg}C^\ast}<\gd$.

Together the above properties conclude that we can choose, for example, $c=4/5<1$ and $\gd = \sqrt{c/M_1}=1/\sqrt{10}$. In particular, the bound is independent of the values of $\gb$ and~$\gg$, as long as they satisfy $\gb+\gg=2$.
As in Theorem~\ref{thm:stabpert} we can now conclude that if $B$ and $C$ are such that for $\gb+\gg=2$ we have $\norm{B}<\gd$, $\norm{C}<\gd$, $\norm{(-A)^{-\gb} B}<\gd$, and $\norm{(-A^\ast)^{-\gg}C^\ast}<\gd$, then the semigroup generated by $A+BC$ is strongly stable.
In particular,
$\norm{CR(\gl,A)B}\leq 4/5<1$ for all $\gl\in \overline{\C^+}\setminus \set{0}$ and $\sup_{\abs{\gw}\leq 1} \abs{\gw}^\ga \norm{R(i\gw,A+BC)}<\infty$.

For rank one perturbations we have $C f = \iprod{f}{c}_{\Lp[2]}$ for a function $c \in \Lp[2](\Omega)$ and $B = b(\cdot) \in \Lp[2](\Omega)$. The perturbed semigroup is strongly stable if $\norm{b}_{\Lp[2]}<1/\sqrt{10}$, $\norm{c}_{\Lp[2]}<1/\sqrt{10}$, 
\eq{
\int_{\Omega} \abs{\mu}^{-2\gb}\abs{b(\mu)}^2 d\mu < \frac{1}{10},
\qquad \mbox{and} \qquad
\int_{\Omega} \abs{\mu}^{-2\gg}\abs{c(\mu)}^2 d\mu < \frac{1}{10}
}
for some $\gb,\gg\geq 0$ satisfying $\gb+\gg=2$.

\section{Conclusions}
\label{sec:conclusions}

In this paper we have studied the preservation of strong stability of a semigroup whose generator has spectrum on the imaginary axis. We have shown that if the growth of the resolvent operator is polynomial near the spectral points $i\gw_k$, then the stability of the semigroup is indeed robust with respect to classes of finite rank perturbations. 

The results concerning the change of the spectrum of $A$ are also valid in the case where the operator $A$ has an infinite number of uniformly separated spectral points on the imaginary axis, and they can also be applied for perturbations that are not of finite rank. However, the additional standing assumptions were required to show the preservation of the uniform boundedness of the semigroup. Therefore, generalizing the conditions on the preservation of uniform boundedness would also immediately improve the results on the preservation of strong stability.


\end{document}